\documentclass[reqno,11pt]{amsart}

\usepackage{amsmath,amssymb,mathtools}
\usepackage{graphicx}
\usepackage[english]{babel}
\usepackage[usenames]{color}
\usepackage{dsfont}
\usepackage{subfigure}
\usepackage{enumerate}

\usepackage{ifpdf}
\ifpdf
\usepackage{hyperref}
\else
\usepackage[hypertex]{hyperref}
\hypersetup{backref, colorlinks=true, bookmarks}
\fi

\newtheorem{theorem}{Theorem}[section]
\newtheorem{lemma}[theorem]{Lemma}
\newtheorem{proposition}[theorem]{Proposition}
\newtheorem{definition}[theorem]{Definition}
\newtheorem{corollary}[theorem]{Corollary}
\theoremstyle{remark}
\newtheorem{remark}[theorem]{Remark}

\numberwithin{equation}{section}

\newcommand{\R}{\mathbb{R}}
\newcommand{\C}{\mathbb{C}}

\newcommand{\N}{\mathbb{N}}
\newcommand{\E}{\mathbb{E}}

\DeclareMathOperator{\erf}{erf}

\newcommand{\ET}[1]{\mathbb{E}_{\Theta}\left[#1\right]}
\newcommand{\PT}[1]{\mathbb{P}_{\Theta}\left[#1\right]}

\newcommand{\Pn}[1]{\mathbb{P}_{\Theta}^n\left[#1\right]}

\newcommand{\la}{\lambda}

\newcommand{\eps}{\epsilon}

\newcommand{\nth}[1]{[t^n]\left[ #1 \right]}
\newcommand{\set}[1]{\left\{#1\right\}}
\DeclareMathOperator{\one}{\mathds{1}}

\renewcommand{\Re}{\mathrm{Re}}

\newcommand{\Sn}{\mathfrak{S}_n}

\DeclareMathOperator{\Li}{Li}

\DeclareMathOperator{\lcm}{lcm}

\begin{document}

\title[The order of large random permutations with cycle weights]{The order of large random permutations with cycle weights.}
\date{\today}

\author[J. Storm]{Julia Storm}
\address{Institut f\"ur Mathematik\\ Universit\"at Z\"urich\\ Winterthurerstrasse 190\\ 8057-Z\"urich,
Switzerland} \email{julia.storm@math.uzh.ch}

\author[D. Zeindler]{Dirk Zeindler}
\address{Lancaster University\\ Mathematics and Statistics \\ Fylde College \\Bailrigg\\Lancaster\\United Kingdom\\LA1 4YF} 
\email{d.zeindler@lancaster.ac.uk}

\begin{abstract}
The order $O_n(\sigma)$ of a permutation $\sigma$ of $n$ objects is the smallest integer $k \geq 1$ such that the $k$-th iterate of $\sigma$ gives the identity. 
A remarkable result about the order of a uniformly chosen permutation is due to  Erd\"os and Tur\'an who proved in 1965 that $\log O_n$ satisfies a central limit theorem. 
We extend this result to the so-called \textit{generalized Ewens measure} in a previous paper. 
In this paper, we establish a local limit theorem as well as, under some extra moment condition, a precise large deviations estimate. 
These properties are new even for the uniform measure. Furthermore, we provide precise large deviations estimates for random permutations with polynomial cycle weights.
\end{abstract}

\maketitle

\tableofcontents

\section{Introduction} \label{section:introduction}

Denote by $\Sn$ the symmetric group, that is the group of permutations on $n$ objects. For a permutation $\sigma \in \Sn$  the order $O_n = O_n(\sigma)$ is defined as the smallest integer $k$ such that the $k$-th iterate of $\sigma$ is the identity. Landau \cite{La09} proved in 1909 that the maximum of the order of all $\sigma \in \Sn$ satisfies, for $n \rightarrow \infty$, the asymptotic
\begin{align*}
 \max_{\sigma \in \Sn}(\log O_n) \sim \sqrt{n \log (n)}.
\end{align*}
On the other hand, $O_n(\sigma)$ can be computed as the least common multiple of the cycle length of $\sigma$. Thus, if $\sigma$ is a permutation that consists of only one cycle of length $n$, then $O_n(\sigma) = \log(n)$ and $(n-1)!$ of all $n!$ permutations share this property. Considering these two extremal types of behavior, the famous result of Erd\"os and Tur\'an \cite{ErTu65} seems even more remarkable: they showed in in 1965 that a uniformly chosen permutation satisfies, as $n \rightarrow \infty$, the central limit theorem
\begin{align}\label{thm:intro_clt_order}
 \frac{\log O_n - \frac{1}{2}\log^2 (n)}{\sqrt{\frac{1}{3}\log^3 (n)}} \overset{d}{\longrightarrow} \mathcal{N}(0,1).
\end{align}
This result was extended to the Ewens measure and to A-permutations, see for instance \cite{AT92b} and \cite{Ya10a}. 

In this paper we study the random variable $\log O_n$ with respect to a weighted measure. We present large deviations estimates and a local limit theorem for $\log O_n$ which are, to our knowledge, new even for the uniform measure. We also give precise expressions for the expected value of $\log O_n$, which extends results from Zacharovas \cite{Za04a}.

The literature on non-uniform permutations has grown quickly in recent years, particularly due to its relevance in mathematical biology and theoretical physics. In this paper, we focus on random permutations with cycle weights as introduced in the recent works of Betz et. al \cite{BeUeVe11} and Ercolani and Ueltschi \cite{ErUe11}. In their model, each cycle of length $m$ is assigned and individual weight $\theta_m \geq 0$. We denote by $C_m=C_m(\sigma)$ the number of cycles of length $m$ in the decomposition of the permutation $\sigma$ as a product of disjoint cycles. The functions $C_1$, $C_2$,\,\dots are random variables on $\Sn$ and we will call them cycle counts.  Then the weighted measure is defined as follows: 

\begin{definition}
\label{def:measure}
Let  $\Theta = \left(\theta_m  \right)_{m\geq1}$ be given, with $\theta_m\geq0$ for every $m\geq 1$. We then define for $\sigma\in \Sn$
\begin{align*}
  \Pn{\sigma}
   \coloneqq
   \frac{1}{h_n n!} \prod_{m=1}^{n} \theta_{m}^{C_m} 
\end{align*}
with $h_n = h_n(\Theta)$ a normalization constant and $h_0 \coloneqq 1$.
If $n$ is clear from the context, we will just write $\mathbb{P}_\Theta$ instead of $\mathbb{P}_\Theta^n$ .
\end{definition}

Notice that special cases of this measure are the uniform measure ($\theta_m =1$) and the Ewens measure ($\theta_m = \theta$). 
Many properties of permutations considered with respect to this weighted measure have been examined for different classes of parameters, 
see for instance \cite{BeUeVe11, ErUe11, HuNaNiZe11,Ma12a,Ma11,MaNiZe11,NiStZe13,NiZe11}.  
Recently, we studied the order of weighted permutations for polynomial parameters $\theta_m = m^{\gamma}$, $\gamma >0$, see \cite{StZe14a}.
We proved that the cycle counts of the cycles of length smaller than a typical cycle in this model can be decoupled into independent Poisson random variables. 
Using this approximation, we extended the Erd\"os-Tur\'an law \eqref{thm:intro_clt_order} to this setting as well as a functional version of it.

In this paper, several properties of $\log O_n$ are considered for two classes of parameters $\Theta = \left(\theta_m  \right)_{m\geq1}$.
 Section~\ref{section:classF} is devoted to generalized Ewens parameters (see Definition~\ref{def_function_class_F_alpha_r} for precise assumptions) and in  Section~\ref{section:algrowth} polynomial parameters $\theta_m = m^\gamma$ with $\gamma >0$ are studied.
 See the respective preliminary sections \ref{subsection:classF_preliminaries} and \ref{subsection:algrowth_preliminaries} for a short overview of the available result for these parameters. 

The challenging point when studying this measure is that due to a lack of compatibility between the different dimensions the Feller coupling is not available for the measure $\mathbb{P}_{\Theta}$. Therefore, new approaches are needed. The crucial feature of $\mathbb{P}_{\Theta}$ is that it is invariant on conjugacy classes. Using generating series and complex analysis methods, a variety of natural properties of weighted random permutations were recently obtained by several authors. The starting point of the study is the relation
\begin{align}\label{eq:def_g_theta}
\sum_{n=0}^\infty h_n t^n = \exp(g_\Theta(t))
\quad \text{ with } \quad 
g_\Theta(t) \coloneqq \sum_{m=1}^\infty \frac{\theta_m}{m}t^m,
\end{align}
where $h_n$ is defined in Definition~\ref{def:measure} and \eqref{eq:def_g_theta} is considered as formal power series in $t$.
Depending on the structure of the $\theta_m$, different methods are required to investigate the asymptotic behavior of $h_n$ and other quantities of interest. It will turn out that for the generalized Ewens parameters the singularity analysis is the right method to choose (see Section~\ref{subsection:classF_preliminaries}) while for polynomial  parameters it is saddle point analysis (see Section~\ref{subsection:algrowth_preliminaries}).

\vskip 40pt 

\section{Generalities}
\label{sec_comb_andgen_of_Sn}


We require in this paper some basic facts about the symmetric group $\Sn$, partitions, and generating functions. 
Since we need precisely the same definitions, notations and tools as in our paper \cite{StZe14a}, we refer the reader to Section 2.1 and Section 2.2 in \cite{StZe14a} (and the references therein).
%
Here, we introduce an important approximation $\log Y_n$ of the random variable $\log O_n$ and we discuss some number theoretic sums which we will encounter frequently throughout the paper.

\vskip 15pt 

\subsection{The approximation random variable $\log Y_n$}
\label{subsec_logYn}

Recall that the order $O_n(\sigma)$ of a permutation $\sigma \in \Sn$ is the smallest integer $k$ such that the $k$-th iterate of $\sigma$ gives the identity. Assume that $\sigma$ decomposes into disjoint cycles $\sigma = \sigma_1\cdots \sigma_\ell$ and denote by $\la_i$ the length of cycle $\sigma_i$. Then $O_n(\sigma)$ can be computed as the least common multiple of the cycle length:
$$O_n(\sigma) = \lcm(\lambda_1, \la_2, \cdots \la_{\ell}).$$
A common approach to investigate the asymptotic behavior of $\log O_n$ is to introduce the random variable
\begin{align}
\label{eq:def_Yn}
 Y_n \coloneqq \prod_{m=1}^{n}m^{C_m}, 
 \quad \text{ that is } \quad
 \log Y_n = \sum_{m=1}^n \log(m) C_m,
\end{align}
where the $C_m$ denote the cycle counts. The basic strategy is to establish results for $\log Y_n$ and then to show that $\log O_n$ and $\log Y_n$ are relatively close in a certain sense. 
To give explicit expressions for $O_n$ and $Y_n$ involving the $C_m$ let us
 introduce
\begin{align}\label{eq:def:classF_Dnk}
D_{nk} \coloneqq \sum_{m=1}^n C_m \one_{\{k | m\}}
\quad \text{ and } \quad
D_{nk}^* \coloneqq \min \{1, D_{nk} \}.
\end{align}
Now let $p_1,p_2,\dots$ be the prime numbers
and $q_{m,i}$ be the multiplicity of a prime number $p_i$ in the number $m$. Then
\begin{align}
\label{eq:Yn_with_Cm}
 Y_n 
&= \prod_{m=1}^{n}m^{C_m} 
= \prod_{m=1}^{n}(p_1^{q_{m,1}} p_2^{q_{m,2}} \cdots p_n^{q_{m,n}})^{C_m}\nonumber\\
&= \prod_{i=1}^{n}p_i^{\,C_1\cdot q_{1,i} + C_2\cdot q_{2,i} + \cdots+ C_n \cdot q_{n,i}}
= 
  \prod_{p \leq n}p^{\,\sum_{j=1}^n{D_{np^j}}} ,
\end{align}
where $\prod_{p \leq n}$ denotes the product over all prime numbers that are less or equal $n$. The last equality can be understood as follows: First, notice that $D_{nk} =0$ for $k>n$. Next, let $p$ be fixed and define $m=p^{\,q_{m,i}} \cdot a$ where $a$ and $p$ are coprime (meaning that their least common divisor is $1$). Then $C_m$ appears exactly once in the sum $D_{np^j}$ if $j\leq q_{m,i}$ but it does not appear if $j> q_{m,i}$. Thus, $C_m$ appears $q_{m,i}$ times in the sum $\sum_{j=1}^n{D_{np^j}}$.

Analogously, we have
\begin{align}
\label{eq:classF_On_with_prims}
O_n = \prod_{p \leq n}p^{\,\sum_{j=1}^n{D^*_{np^j}}} .
\end{align}

To simplify the logarithm of the expressions \eqref{eq:Yn_with_Cm} and \eqref{eq:classF_On_with_prims},  we introduce the von Mangoldt function $\Lambda$, which is defined as
\begin{align}\label{eq:intro_Mangoldt}
\Lambda(n) = 
\begin{cases} 
\log(p) &\mbox{if } n = p^k \text{ for some prime } p \text{ and } k\geq 1,  \\ 
0 & \mbox{otherwise.} 
\end{cases} 
\end{align}
Consequently,
\begin{align}
\label{eq:log_On_von_mangold}
\log Y_n= \sum_{k \leq n} \Lambda(k) D_{nk} 
\quad \text{ and } \quad
\log O_n= \sum_{k \leq n} \Lambda(k) D^*_{nk}.
\end{align}

Now define 
\begin{align}\label{eq:intro_Delta_n}
\Delta_n 
\coloneqq \log Y_n - \log O_n 
 = \sum_{k \leq n}  \Lambda(k) \big(D_{nk} - D^*_{nk}\big) .
  \end{align}
In order to prove properties of $\log O_n$ they are first established for $\log Y_n$ and then one needs to show that $\Delta_n$ is approximately small enough to transfer the result to $\log O_n$, see for example Lemma~\ref{lem:classF_closeness1 log Y_n and log O_n} and Lemma~\ref{lem:algrowth_closeness}.

An important tool to study $\log Y_n$ is its moment generating function. By using a randomized version of the measure $\mathbb{P}_{\Theta}$, one can show 
\begin{align}\label{eq:generating_series_logYn_allgemein}
 \sum_{n=0}^{\infty} h_n \E_{\Theta}[\exp(s \log Y_n)] t^n  
= \exp\left( \sum_{m=1}^{\infty} \frac{\theta_m}{m^{1 - s}} t^m  \right),
\end{align}

see Lemma 2.7 and equation (2.6) in \cite{StZe14a}.



\vskip 15pt 
\subsection{Number theoretic sums}
We recall the asymptotic behavior of some averages over multiplicative functions involving the von Mangoldt function $\Lambda$, which will be particularly useful to study the difference of $\log O_n$ and $\log Y_n$, see \eqref{eq:intro_Delta_n}. 
Let us begin with the Chebyshev function $\psi$, which is defined as
\begin{align}\label{eq:intro_Chebychev}
 \psi(x) \coloneqq \sum_{k \leq x} \Lambda(k) = \sum_{p^k \leq x} \log(p).
\end{align}
By definition, the prime number theorem is equivalent to
\begin{align}\label{eq:Chebychev_asymp}
\psi(x) = x \big(1+o(1) \big) \quad \quad \text{ as } x \rightarrow \infty.
\end{align}
A more precise explicit formula which was proved by Mangoldt is given by 
\begin{align}\label{eq:Chebychev_asymp-exact}
 \psi(x) = x - \sum_{\rho} \frac{x^{\rho}}{\rho} - \log(2\pi) - \frac{1}{2}\log(1-x^{-2}),
\end{align}
where the sum is taken over the zeros of the Riemann zeta function (see \cite[Section II.4.3]{Te95}). Then  the Riemann hypothesis is equivalent to 
\begin{align}\label{eq:Chebychev_asymp_RH}
 \psi(x) = x + O(x^{1/2 + \epsilon}) \quad \text { for all } \epsilon >0,
\end{align}
see \cite[Section II.4, Corollary 3.1]{Te95}. The relation of $\psi(n)$ and the least common multiple of the numbers $1,2,...,n$ is given by
\begin{align*}
\lcm(1,2,...,n) = \exp(\psi(n)).
\end{align*}

Furthermore, by \cite[Theorem 4.9]{Ap84},
\begin{align}\label{eq:intro_Mangoldt_sum1}
\sum_{k\leq x} \frac{\Lambda (k)}{k} = \log(x) + O(1) \ \text{ as }x\to\infty
\end{align}
holds, and by \eqref{eq:Chebychev_asymp} this can be generalized for $0\neq \alpha\neq 1$ to
\begin{align}\label{eq:intro_Mangoldt_sum3}
\sum_{k=x}^y \Lambda(k) k^{-\alpha} 
&=
\sum_{k = x}^y \Lambda(k) \int_k^y \alpha t^{-\alpha-1} dt + y^{-\alpha} \sum_{k = x}^y \Lambda(k) \nonumber\\
&= \alpha \int_{x}^y \sum_{k=x}^t \Lambda(k) t^{-\alpha-1} dt + y^{-\alpha}(y - x +o(y)) \nonumber\\
&= \frac{1+\alpha}{1-\alpha} \big(y^{1-\alpha} - x^{1-\alpha}\big)(1+o(1)).
\end{align}
%

%
%
%
Finally, recall also the Euler-Maclaurin formula
\begin{align}\label{eq:intro_euler_maclaurin}
\sum_{m=1}^b f(m)
= \int_0^b f(x) dx + \int_0^b (x-\lfloor x \rfloor)f'(x)dx + f(b)(b-\lfloor b \rfloor)  .
\end{align}

\vskip 40pt 

\section{The generalized Ewens measure}
\label{section:classF}

The first class of parameters $\Theta = (\theta_m)_{m \geq 1}$ of interest are the so-called generalized Ewens parameters. Roughly speaking, this class comprises all types of parameters such that the generating series $g_{\Theta}$ as defined in \eqref{eq:def_g_theta} exhibits logarithmic singularities. To make this notion precise, we consider $\Theta = (\theta_m)_{m \geq 1}$ such that $g_{\Theta}$ belongs to the set $\mathcal{F}(r,\vartheta,K)$, see Definition~\ref{def_function_class_F_alpha_r} below.  This class of parameters was recently studied by several authors. The case $\theta_m \rightarrow \vartheta$ (which corresponds to $\mathcal{F}(1,\vartheta,K)$) was studied for example in \cite{BeUeVe11}, where results on the length of a typical cycle and the expected value of the total number of cycles are obtained. In \cite{NiZe11} a central limit theorem and Poisson approximation estimates for the total number of cycle are proved for the general case  $\mathcal{F}(r,\vartheta,K)$. These results where complemented in \cite{NiStZe13}, where the behavior of large cycles was studied and a functional central limit theorem for the cycle counts was obtained.

In all these works it turns out that the behavior of weighted random permutations with parameters corresponding to $\mathcal{F}(r,\vartheta,K)$ 
(almost) coincides with that of permutations considered with respect to the Ewens measure with parameter $\vartheta$. 
It it thus natural to expect that the Erd\"os-Tur\'an law as stated in \eqref{thm:intro_clt_order} should also be valid for parameters of the class $\mathcal{F}(r,\vartheta,K)$. 
This is indeed true, as we will show in Theorem~\ref{thm:classF_logOn_central_limit}. 
Furthermore, we will present results about the order of weighted random permutations that are even new for the Ewens measure, 
such as a local limit theorem (see Section~\ref{subsection:classF_mod_conv_local_limit}) and large deviations estimates (see Section~\ref{subsection:classF_large_dev}).

\vskip 15pt 
\subsection{Preliminaries}\label{subsection:classF_preliminaries}
To determine the framework of this section the following preliminary definition is needed.

\begin{definition}
\label{def_delta_0}
Let $0< r < R$ and $0 < \phi <\frac{\pi}{2}$ be given. We then define
\begin{align}
\Delta_0 = \Delta_0(r,R,\phi) = \set{ z\in \C ; |z|<R, z \neq r ,|\arg(z-r)|>\phi}
\label{eq_def_delta_0}.
\end{align}

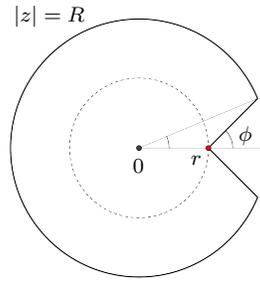
\begin{figure}[ht!]
\centering
 \includegraphics[height=.18\textheight]{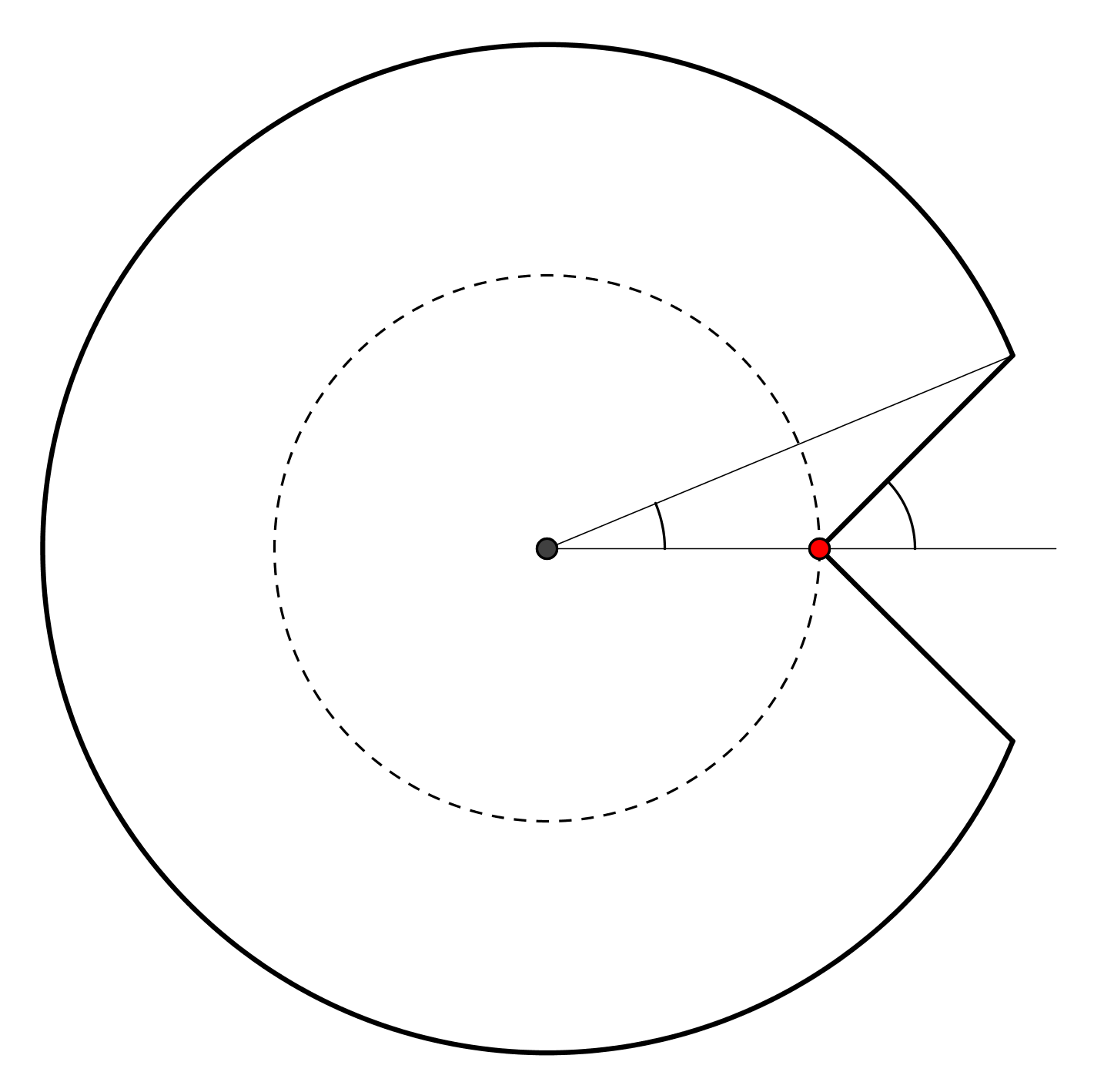}
\put(-55,43){\mbox{\scriptsize$0$}}
\put(-33,46){\mbox{\scriptsize$r$}}
\put(-15,55){\mbox{\scriptsize$\phi$}}
\put(-100,100){\mbox{\scriptsize$|z|=R$}}
\caption{Illustration of $\Delta_0$}
\label{fig_delta_0}
\end{figure}

\end{definition}
Let us now introduce the \textit{generalized Ewens measure}. Rather than defining conditions for the parameters $\Theta = (\theta_m)_{m \geq 1}$ directly, we impose them on the generating series $g_{\Theta}$. We require that $g_{\Theta}$ is analytic in a $\Delta_0$-domain and that it admits logarithmic growth at its dominant singularity. 
\begin{definition}
 \label{def_function_class_F_alpha_r}
Let $r,\vartheta>0$ and $K\in\R$ be given. We write $\mathcal{F}(r,\vartheta,K)$ for the set of all functions $g$ satisfying
\begin{enumerate}
 \item[\textup{(1)}]  $g$ is holomorphic in $\Delta_0(r,R,\phi)$ for some $R>r$ and $0 < \phi <\frac{\pi}{2}$,
 \item[\textup{(2)}]  
   \begin{align}
   g(t) = \vartheta \log\left( \frac{1}{1-t/r} \right) + K + O \left( t-r \right)   \text{ as } t\to r.
   \label{eq_class_F_r_alpha_near_r}
   \end{align}
\end{enumerate}
\end{definition}
Notice that $\theta_m = \vartheta$ leads to $g_{\Theta}(t) = - \vartheta \log(1-t) \in \mathcal{F}(1, \vartheta,0)$ and thus the Ewens measure is covered by the family $\mathcal{F}(r, \vartheta,K)$. More generally, functions of the form 
$g_{\Theta}(t) = - \vartheta \log(1-t) + f(t)$ with $f$ holomorphic for $|t| < 1 + \epsilon$ are contained in $\mathcal{F}(1, \vartheta,f(1))$. In particular, the case $\theta_m \neq \vartheta$ for only finitely many k is included in $\mathcal{F}(1, \vartheta, . )$.

\begin{remark}
The justification for the name \textit{generalized Ewens measure} relies on the following observation.
Theorem VI.3 and VI.4 in \cite{FlSe09} implies that if $g_{\Theta}(t)$ is defined as in \eqref{eq:def_g_theta} and the parameters $\theta_m$ are such that 
$g_{\Theta}$ belongs to $\mathcal{F}(r,\vartheta,K)$, then there exists some $\eps_m$ such that 
 \begin{align}
 \label{eq:classF_theta_k_to_vartheta}
\theta_m r^m = \vartheta +\eps_m \quad \text{ with }\ |\eps_m|=O(1) \ \text{ and }\ \sum_{m=1}^\infty \frac{|\eps_m|}{m} < \infty.
\end{align}
%
\end{remark}
Notice that there are examples in $\mathcal{F}(r,\vartheta,K)$ with $|\eps_m| \not\to 0$. We will occasionally assume that  $|\eps_m| \to 0$ to get nicer results.

With these assumptions on the generating series at hand, we can compute the asymptotic behavior of $h_n$.
\begin{corollary}[\cite{NiStZe13}, Corollary 3.4]
\label{cor:haviour_of_hn}
Let $g_\Theta(t)$ in $\mathcal{F}(r,\vartheta, K)$ be given, then
\begin{align*}
 h_n 
=
\frac{n^{\vartheta -1} e^{K } }{r^n \Gamma(\vartheta)} \left(1 +  O\left(\frac{1}{n} \right) \right).
\end{align*}
\end{corollary}
The starting point of our study of the properties of $\log O_n$ is the closeness of $\log O_n$ and $\log Y_n$. Recall $\Delta_n$ defined in \eqref{eq:intro_Delta_n}. 
\begin{lemma}\label{lem:classF_closeness1 log Y_n and log O_n}
Let $(\theta_m)_{m \geq 1}$ be such that $g_{\Theta} \in \mathcal{F}(r,\vartheta,K)$. Then, as $n \to \infty$, the following asymptotic holds for every constant $\kappa$: 
\begin{align*}
 \PT{\Delta_n \geq \log (n) (\log \log (n))^{\kappa}} 
 =O\big((\log \log (n))^{1-\kappa}\big).
\end{align*}
\end{lemma}
The analogue result for the Ewens measure was proved in \cite{DePi85}. In Section~\ref{subsection:classF_E_On} we will present a much more precise expression for $\ET{\Delta_n}$.
For the proof of Lemma~\ref{lem:classF_closeness1 log Y_n and log O_n} the following proposition is required.

\begin{proposition}
\label{prop:bound_Dnk} 
Suppose that $g_{\Theta}$ belongs to $\mathcal{F}(r,\vartheta,K)$. Then  
\begin{enumerate}
\item[\textup{(1)}]  
$\ET{D_{nk}} 
= O\left( \frac{\log(n)}{k} + n^{-\theta}\one_{\{k | n\}} \right),$
\item[\textup{(2)}]  
$ \ET{D_{nk}(D_{nk} - 1)}
= O\left(\frac{\log^2 (n)}{k^2}+ n^{-2\theta}\one_{\{k | n\}}\right). $
\end{enumerate}
Furthermore, the error terms are uniform in $k$ for $1\leq k\leq n$.
\end{proposition}

\begin{proof}[Proof of Lemma \ref{lem:classF_closeness1 log Y_n and log O_n}]
Notice that $\Delta_n$ defined in \eqref{eq:intro_Delta_n} can be estimated as
\begin{align*}
 \Delta_n = \sum_{k=1}^n \Lambda(k) \big(D_{nk} - D^*_{nk}\big) =: \sum_{k=1}^n \Lambda(k) \, \Delta_{nk}
\end{align*}
with
\begin{align*}
 \Delta_{nk} \leq D_{nk} \quad \text{and} \quad \Delta_{nk} \leq D_{nk}(D_{nk} - 1).
\end{align*}
Thus
\begin{align*}
 \ET{\Delta_n}
&= \sum_{k=1}^n \Lambda(k) \, \E_{\Theta}[\Delta_{nk}] \nonumber \\ 
&\leq \sum_{k=1}^{\lfloor \log (n) \rfloor } \Lambda (k) \, \E_{\Theta}[D_{nk}] + \sum_{k= \lceil \log (n) \rceil }^n \Lambda(k) \, \E_{\Theta}[D_{nk}(D_{nk} - 1)]. 
\end{align*}

Then Proposition~\ref{prop:bound_Dnk} together with \eqref{eq:intro_Mangoldt_sum1} and \eqref{eq:intro_Mangoldt_sum3} gives
\begin{align}\label{eq:classF_expectation_Delta_n}
 \ET{\Delta_n}
&= O\bigg( \log (n) \sum_{k=1}^{\lfloor \log(n) \rfloor } \frac{\Lambda (k)}{k} + 
\log^2 (n) \sum_{k= \lceil \log (n) \rceil }^n \frac{\Lambda (k)}{k^2} \bigg) \nonumber\\
&= O(\log (n) \log\log (n)).
\end{align}
%

Now Chebychev's inequality implies for $n \to\infty$
\begin{align*}
 \PT{\Delta_n \geq \log (n) (\log \log (n))^{\kappa}} 
\leq \frac{\E_{\Theta}[\Delta_n]}{\log (n)(\log \log (n))^{\kappa}}
=O\big((\log \log (n))^{1-\kappa}\big)
\end{align*}
and this completes the proof of the lemma. 
\end{proof}

\begin{proof}[Proof of Proposition~\ref{prop:bound_Dnk} ]
We begin with $(1)$. Lemma 2.5 in \cite{StZe14a} and \eqref{eq:def:classF_Dnk} yield
%
\begin{align*}
 \ET{D_{nk}} 
 = 
 \sum_{m=1}^n \ET{C_m} \one_{\{k | m\}} 
 = 
 \sum_{m=1}^n \frac{\theta_m}{m} \one_{\{k | m\}} \frac{h_{n-m}}{h_n}.
\end{align*}
We have to distinguish the cases $\vartheta \geq 1$ and $\vartheta < 1$, see \eqref{eq:classF_theta_k_to_vartheta}.
If $\vartheta \geq 1$, then it follows with with Corollary~\ref{cor:haviour_of_hn} and \eqref{eq:classF_theta_k_to_vartheta} that $ \theta_m h_{n-m}/h_n$ is bounded and thus
\begin{align*}
  \ET{D_{nk}} 
  = 
  O\left( \sum_{m=1}^n \frac{1}{m} \one_{\{k | m\}} \right)
  =
  O\left( \frac{1}{k} \sum_{j=1}^{n/k} \frac{1}{j} \right)
   =
  O\left( \frac{\log(n)}{k}  \right).
\end{align*}
If $\vartheta <1$, we have to be more careful. We get again with \eqref{eq:classF_theta_k_to_vartheta} and Corollary~\ref{cor:haviour_of_hn} 
\begin{align*}
\ET{D_{nk}} 
=
O\left( \sum_{m=1}^{n-1} \frac{1}{m} \one_{\{k | m\}}  \left(1-\frac{m}{n}\right)^{\theta-1} + n^{-\theta}\one_{\{k | n\}}\right),
\end{align*}
%
%
where
\begin{align*}
 \sum_{m=1}^{n-1} \frac{1}{m} \one_{\{k | m\}}  \left(1-\frac{m}{n}\right)^{\theta-1}
&=
O\left( \sum_{m=1}^{n/2} \frac{\one_{\{k | m\}} }{m}   + \frac{1}{n}\sum_{m>n/2}^{n-1}\one_{\{k | m\}}  \left(1-\frac{m}{n}\right)^{\theta-1}  \right)\\
&=
O\left( \frac{\log(n)}{k} + \frac{1}{n}\int_{n/(2k)}^{(n-1)/k} \left(1-\frac{kx}{n}\right)^{\theta-1} \,dx\right)\\
&=
O\left( \frac{\log(n)}{k} +  \frac{1}{k} \right) = O\left( \frac{\log(n)}{k}\right).
\end{align*}
This completes the proof of $(1)$. Furthermore,
\begin{eqnarray*}
 \ET{D_{nk}(D_{nk} - 1)}
= \ET{\bigg(\sum_{m=1}^n C_m \one_{\{k | m\}}\bigg) \bigg(\sum_{m=1}^n C_m \one_{\{k | m\}} - 1\bigg)}\\
= \ET{\sum_{m,m'=1}^n C_m C_{m'} \one_{\{k|m \, ; \, k|m'\}} - \sum_{m=1}^n C_m \one_{\{k | m\}}}\\
= \ET{\sum_{\substack{m,m'=1 \\ m\neq m'}}^n C_m C_{m'} \one_{\{k|m \, ; \, k|m'\}} + \sum_{m=1}^n C_m (C_m - 1) \one_{\{k | m\}} }\\
= 
\sum_{\substack{m,m'=1 \\ m\neq m'}}^n \frac{\theta_m}{m}\frac{\theta_{m'}}{m'} \one_{\{k|m \, ; \, k|m'\}} \frac{h_{n-m-m'}}{h_n}
+
\sum_{m=1}^n \left(\frac{\theta_m}{m}\right)^2 \one_{\{k | m\}}\frac{h_{n-2m}}{h_n}.
\end{eqnarray*}
%
%
A similar argument as for $\ET{D_{nk}}$ gives the upper bound in $(2)$. 
\end{proof}

With Lemma~\ref{lem:classF_closeness1 log Y_n and log O_n} at hand, one can directly deduce the Erd\"os-Tur\'an law as it was stated in \eqref{thm:intro_clt_order} for uniform random permutations.

\begin{theorem}\label{thm:classF_logOn_central_limit}
Suppose that $g_\Theta(t)$ belongs to $\mathcal{F}(r,\vartheta, K)$, then
\begin{align*}
 \frac{\log O_n - \frac{\vartheta}{2}\log^2(n)}{\sqrt{\frac{\vartheta}{3}\log^3(n)}} \overset{d}{\longrightarrow} \mathcal{N}(0,1), 
\end{align*}
where $\mathcal{N}(0,1)$ denotes a standard Gaussian random variable.
\end{theorem}
\begin{proof}
Given Lemma~\ref{lem:classF_closeness1 log Y_n and log O_n}, it suffices to show the required asymptotic holds for $\log Y_n$. In a beautiful proof,  DeLaurentis and Pittel \cite{DePi85} deduce this for the uniform measure from a functional version of the central limit theorem for the cycle counts. The analogue result for the generalized Ewens measure was proved in \cite[Theorem 5.5]{NiStZe13}. The rest of the proof is completely similar to the proof in \cite{DePi85}.
\end{proof}

\vskip 15pt 

\subsection{The truncated order}
\label{subsection:classF_truncated_order}

To establish further properties of the order of weighted permutations, it turns out to be convenient to introduce truncated versions of $\log Y_n$ and $\log O_n$ in order to simplify computations: 
\begin{align}
\label{eq:def_On_tilde_Yn_tilde}
 \widetilde{O}_n = \lcm\{m \leq b_n; \, C_m\neq0\} \quad \text{ with }\quad b_n\coloneqq n/\log^2(n)
\end{align}
and similarly
\begin{align*}
\widetilde{Y}_n \coloneqq \prod_{m=1}^{b_n} m^{C_m} .
\end{align*}
The advantage of the truncated variables is that less analytic assumptions on the sequence of parameters $(\theta_m)_{m\geq 1}$ are required and that many computations are simpler; see also Remark~\ref{rem:why_not_Yn}. Nonetheless, $\widetilde{Y}_n$ and $\widetilde{O}_n$ share many important properties with $Y_n$ and $O_n$.
Similarly to \eqref{eq:log_On_von_mangold} we have
\begin{align}
\label{eq:log_tildeYn_von_mangold}
\log \widetilde{Y}_n &= \sum_{k \leq n} \Lambda(k) \widetilde{D}_{nk} 
\quad \text{ with }\quad 
\widetilde{D}_{nk} \coloneqq \sum_{m=1}^{b_n} C_m \one_{\{k | m\}},\\
\label{eq:log_tildeOn_von_mangold}
\log \widetilde{O}_n &= \sum_{k \leq n} \Lambda(k) \widetilde{D}^*_{nk}
\quad \text{ with }\quad 
\widetilde{D}^*_{nk} \coloneqq \min\{1,\widetilde{D}_{nk}\}.
\end{align}
%
%
%
%

Our basic strategy is as follows. We will establish properties of $\log\widetilde{Y}_n$ and transfer them to $\log\widetilde{O}_n$ and finally to $\log O_n$. For the first transfer, define
$$
\widetilde{\Delta}_n\coloneqq
\log \widetilde{Y}_n -\log \widetilde{O}_n$$
and notice that $0\leq \widetilde{\Delta}_n \leq \Delta_n$. Thus, Lemma~\ref{lem:classF_closeness1 log Y_n and log O_n} yields
\begin{align}\label{eq:classF_Delta_tilde}
 \PT{\widetilde{\Delta}_n \geq \log(n)(\log\log(n))^{\kappa}}
&= O\big((\log\log (n))^{1-\kappa}\big).
\end{align}
For the second transfer, notice that
\begin{align*}
\log O_n - \log \widetilde{O}_n 
\leq \log Y_n - \log \widetilde{Y}_n
= \sum_{m=b_n +1}^n \log(m) \, C_m = O(\log(n)\log\log(n)).
\end{align*}
In order to study $\log\widetilde{Y}_n$, we need its moment generating function.
\begin{lemma} \label{lem:classF_generating_E_log_Y_tilde}
Let $g_\Theta(t)$ be as in \eqref{eq:def_g_theta} and $s\in\C$, then

\begin{enumerate}
\item[\textup{(1)}]  $\displaystyle \ET{\log \widetilde{Y}_n}
= \frac{1}{h_n} \nth{ \left(\sum_{m=1}^{b_n} \log(m)\frac{\theta_m}{m}t^m\right) \exp\left(g_\Theta(t) \right)},$

\item[\textup{(2)}]  $\displaystyle \ET{e^{s\log \widetilde{Y}_n}}
= \frac{1}{h_n} \nth{  \exp\left(g_\Theta(t) + \left(\sum_{m=1}^{b_n} (e^{s \log(m)}-1)\frac{\theta_m}{m}t^m\right)\right)}, $
\end{enumerate}
where the functions on the right-hand sides are considered as formal power series in $t$.
\end{lemma}
\begin{proof}
Equation $(1)$ follows from $(2)$ by differentiating once with 
respect to $s$ and substituting $s=0$. We thus only have to prove $(2)$.
For this, let $c\in \N$ be fixed and consider $Y_n^c \coloneqq \prod_{m=1}^c m^{C_m}$.
 We now apply the so called cycle index theorem with the formulation in 
Lemma 2.3 in \cite{StZe14a} with $a_m= e^{s \log(m)} \theta_m$ for $m\leq c$ and $a_m=\theta_m$ for $m>c$. 

We then have as formal power series
\begin{align*}
 \sum_{n=0}^\infty h_n t^n \ET{e^{s\log Y_n^c} }
&=
 \sum_{n=0}^\infty \frac{t^n}{n!} \sum_{\sigma\in\Sn} \prod_{m=1}^c (e^{s\log(m)} \theta_m)^{C_m} \prod_{m=c+1}^\infty \theta_m^{C_m}\nonumber\\
&=
\exp\left( \sum_{m=1}^c e^{s\log(m)} \frac{\theta_m}{m} t^m + \sum_{m=c+1}^\infty \frac{\theta_m}{m} t^m \right)\nonumber\\
&=
\exp\left(g_\Theta(t) + \left(\sum_{m=1}^c (e^{s\log(m)}-1)\frac{\theta_m}{m}t^m\right)\right).
\end{align*}
Now identify the coefficients of $t^n$ on both sides and obtain
\begin{align*}
\ET{e^{s\log Y_n^c}}
&=
\frac{1}{h_n} \nth{
\exp\left(g_\Theta(t) + \left(\sum_{m=1}^c (e^{s\log(m)}-1)\frac{\theta_m}{m}t^m\right)\right)}.
\end{align*}
Equation (2) now follows by substituting $c=b_n$. 
\end{proof}
The previous lemma yields
\begin{lemma}
\label{lem:E_tilde_Yn}
If $g_\Theta$ belongs to $\mathcal{F}(r,\vartheta, K)$, then 
\begin{align*}
\ET{\log \widetilde{Y}_n}
=
\sum_{m=1}^{b_n} \frac{\log(m)}{m} \theta_m r^m + O\left(\log^{-1}(n) \right) .
\end{align*}
Furthermore we get for $s\in\C$
\begin{align*}
 \ET{e^{s\frac{\log \widetilde{Y}_n}{\log(n)}}}
&=
\exp\left(
\sum_{m=1}^{b_n} \left(e^{s\frac{\log(m)}{\log(n)}}-1\right)\frac{\theta_m}{m}r^m
\right)
\left(1+O\left(n^{-1}\right)  \right)
\end{align*}
and the error term is uniform in $s$ for $s$ bounded.
\end{lemma}
\begin{proof}
We use Lemma~\ref{lem:classF_generating_E_log_Y_tilde} and get with Cauchy's integral formula
\begin{align*}
 h_n \ET{\log \widetilde{Y}_n}
&= 
\frac{1}{2 \pi i} \int_{\gamma}  \widetilde{q}_1(t)\exp\left(g_\Theta(t) \right) \frac{dt}{t^{n+1}},\\
%
 h_n \ET{e^{s\frac{\log \widetilde{Y}_n}{\log(n)}}}
&= 
\frac{1}{2 \pi i} \int_{\gamma}  \widetilde{e}(s,t)\exp\left(g_\Theta(t) \right) \frac{dt}{t^{n+1}}
\end{align*}
where $\gamma$ is a simple closed curve around $0$ and 
\begin{align*}
 \widetilde{q}_1(t) \coloneqq \sum_{m=1}^{b_n} \log(m)\frac{\theta_m}{m}t^m, 
 \quad \quad 
 \widetilde{e}(s,t) \coloneqq \sum_{m=1}^{b_n} \big(e^{s\frac{\log(m)}{\log(n)}}-1\big)\frac{\theta_m}{m}t^m.
\end{align*}
By assumption, $g_\Theta$ is analytic in a domain $\Delta_0=\Delta(r,R,\phi)$; see Definition~\ref{def_delta_0}. 
We choose for both integrals the curve $\gamma$ as in Figure~\ref{fig:curve_flajolet_4}, such that $\gamma$ is contained in the $\Delta_0$-domain.
More precisely, we choose the radius of the big circle $\gamma_4$ as $R^\prime\coloneqq r(1+b_n^{-1})$ with $b_n$ as in \eqref{eq:def_On_tilde_Yn_tilde}, the radius of the small circle as $1/n$ and 
the angle of the line segments independent of $n$.
Notice that $\widetilde{q}_1(t)$ and $\widetilde{e}(s,t)$ are for given $n$ polynomials and we thus do not require any further analytic assumptions to use this curve.
%
%
%
%
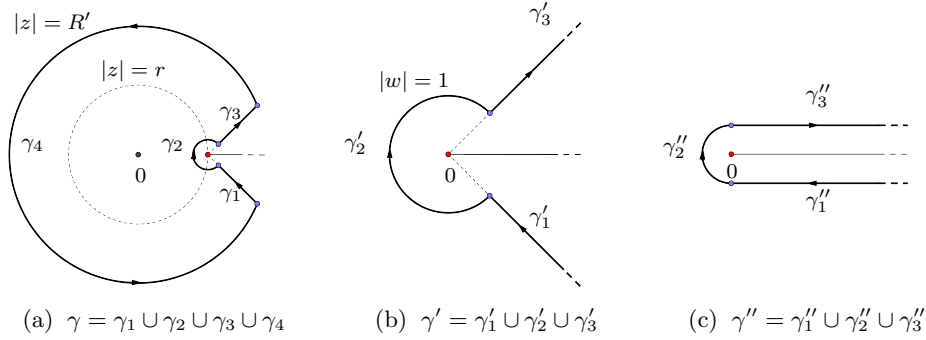
\begin{figure}[h]
\centering
\subfigure[\,$\gamma=\gamma_1\cup\gamma_2\cup\gamma_3\cup\gamma_4$]{\hspace{-1pc}
\includegraphics[height=.18\textheight]{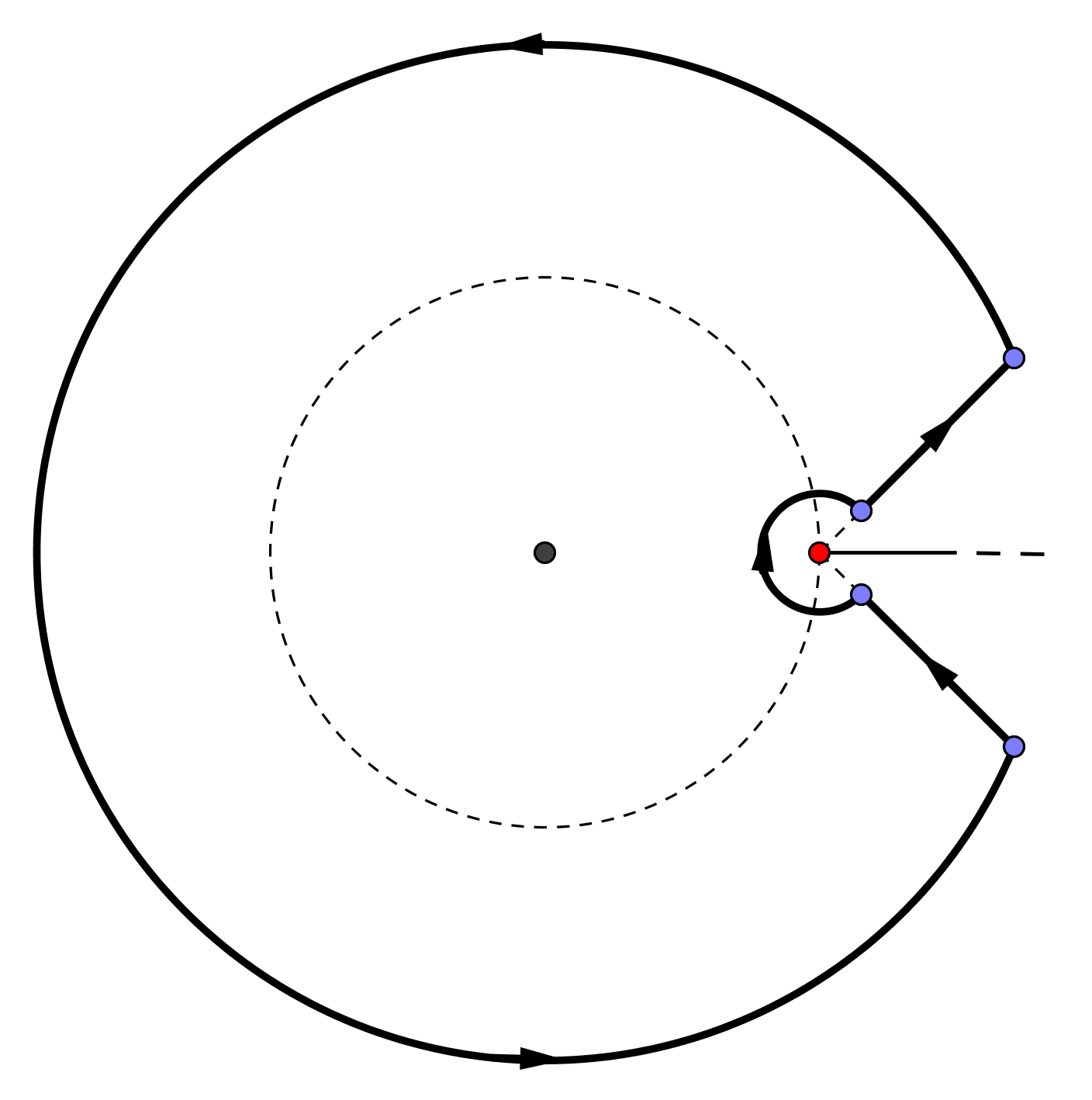}
 \label{fig:curve_flajolet_4}
 \put(-57,42){\mbox{\scriptsize$0$}}
  \put(-70,82){\mbox{\scriptsize$|z|= r$}}
  \put(-103,100){\mbox{\scriptsize$|z|= R^\prime$}}
  \put(-47,55){\mbox{\scriptsize$\gamma_2$}}
   \put(-25,67){\mbox{\scriptsize$\gamma_3$}}
   \put(-25,36){\mbox{\scriptsize$\gamma_1$}}
   \put(-100,55){\mbox{\scriptsize$\gamma_4$}}
 }
\hspace{2.5pc}\subfigure[\,$\gamma'=\gamma'_1\cup\gamma'_2\cup\gamma'_3$]{
   \includegraphics[height=.18\textheight]{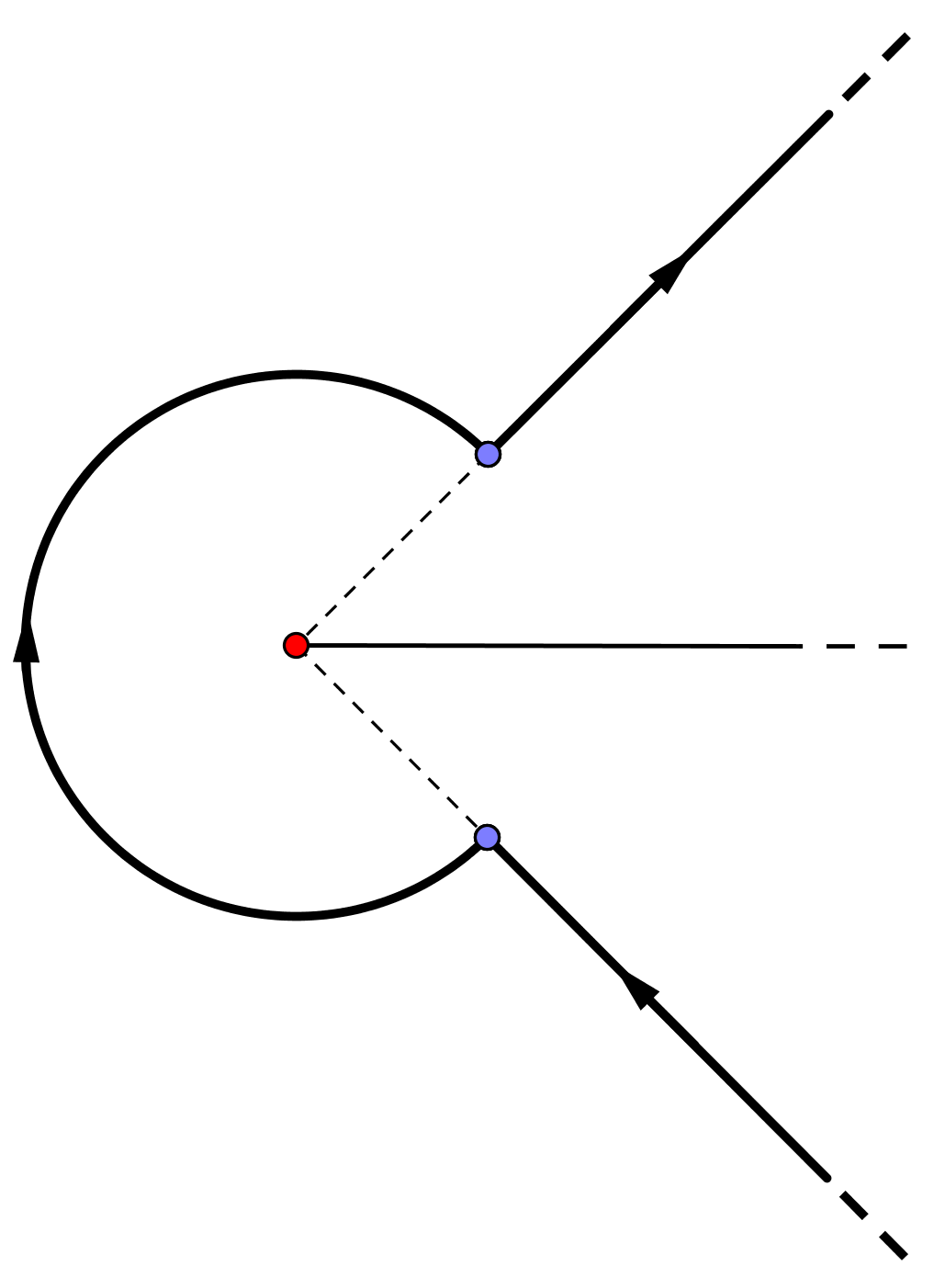}
   \label{fig:gamma'}
   \put(-58,42){\mbox{\scriptsize$0$}}
   \put(-83,78){\mbox{\scriptsize$|w| =1$}}
    \put(-26,27){\mbox{\scriptsize$\gamma'_1$}}
   \put(-96,55){\mbox{\scriptsize$\gamma'_2$}}
   \put(-26,104){\mbox{\scriptsize$\gamma'_3$}}
 }
\hspace{2.5pc}\subfigure[\,$\gamma''=\gamma''_1\cup\gamma''_2\cup\gamma''_3$]{
   \includegraphics[height=.18\textheight]{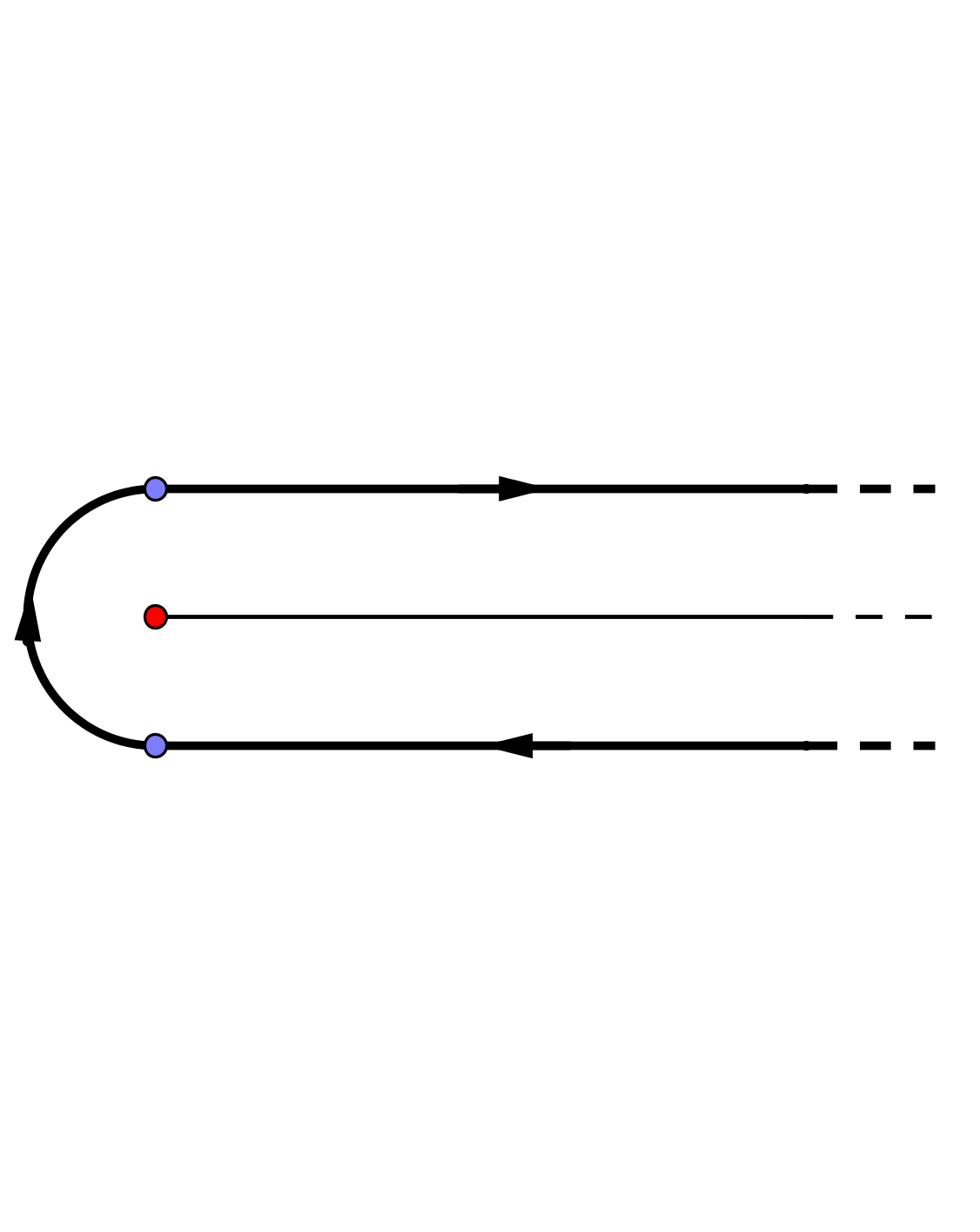}
   \label{fig:gamma''}
    \put(-76,43,5){\mbox{\scriptsize$0$}}
    \put(-46,33){\mbox{\scriptsize$\gamma''_1$}}
    \put(-100,53.5){\mbox{\scriptsize$\gamma''_2$}}
    \put(-46,73){\mbox{\scriptsize$\gamma''_3$}}
 }
\caption{{The curves used in the proof of Lemma~\ref{lem:E_tilde_Yn}.}}
\end{figure}
First, consider the integral over the big circle $\gamma_4$ and show that its contribution is negligible.
We get with \eqref{eq:classF_theta_k_to_vartheta} and for $\varphi\in[-\pi,\pi]$ 
\begin{align*}
 \left |\widetilde{q}_1(R'e^{i\varphi}) \right|
&=
O \left(\sum_{m=1}^{b_n} \frac{\log(m) }{m} \left(1+b_n^{-1}\right)^m  \right) \nonumber\\
 &=
O \left(\sum_{m=1}^{b_n} \frac{\log(m) }{m} \left(1+ O\left(m b_n^{-1}\right)\right)  \right)\nonumber\\
&=
O \left(\sum_{m=1}^{b_n} \frac{\log(m) }{m}  \right)
=
O( \log^2(b_n))
=
O( \log^2(n)).
\end{align*}
We have used that $m \leq b_n$ and thus $m\log (1+b_n^{-1})= \frac{m}{b_n}(1+o(1))$.
Since $(e^{\frac{s \log(m)}{\log(n)}}-1)$ is bounded for $s$ bounded, we can apply for $\widetilde{e}(s,t)$ the same estimate as for $\widetilde{q}_1$ and get
\begin{align*}
 \left |\widetilde{e}(s,R'e^{i\varphi}) \right|
&=
O( \log^2(b_n))
=
O( \log^2(n)).
\end{align*}
Furthermore, we have on the $\Delta_0$-domain
\begin{align*}
 |g_\Theta(t)| \leq \vartheta \log \left|\frac{1}{1-t/r}\right| + O(1)
 \ \implies \
|g_\Theta(R^\prime e^{i\varphi})| 
\leq
\vartheta \log(b_n) +O(1).
\end{align*}
Finally, 
\begin{align*}
  (R^\prime)^{-n}
&=
r^{-n}\big(1+n^{-1}\log^2(n)\big)^{-n} = r^{-n} \exp\big(-\log^2(n) + O(\log^4(n)/n)\big)\nonumber\\
&= 
O\bigl(r^{-n} \exp(-\log^2(n))\bigr). 
\end{align*}
Combining these three estimates yields
\begin{align*}
\left| 
\frac{1}{2 \pi i} 
\int_{\gamma_4}  \widetilde{q}_1(t)\exp\left(g_\Theta(t) \right) \frac{dt}{t^{n+1}} \right|
=
O\big(r^{-n} n^\vartheta \exp(-\log^2(n))\big).
\end{align*}
Since $h_n  \sim e^{K} n^{\vartheta-1} \bigl(\Gamma(\vartheta) r^n\bigr)^{-1}$ (see Corollary~\ref{cor:haviour_of_hn}), we can neglect the integral over $\gamma_4$ 
with respect to the scale of the problem.
Let us consider the remaining parts of the curve. The computations of the integrals over $\gamma_1,\gamma_2$ and $\gamma_3$ are 
completely similar to the computations in the proof of Theorem~VI.3 in \cite{FlSe09}. We thus give only a short overview. 
We start with $\widetilde{q}_1$ and write $t= r(1+wn^{-1})$ with $w= O(\log^2(n))$ and obtain
\begin{align}
 \widetilde{q}_1\left(r+\frac{rw}{n}\right)
&=
\sum_{m=1}^{b_n}  \frac{\log(m)}{m} \theta_m r^m \left(1+\frac{w}{n}\right)^m
=
\sum_{m=1}^{b_n}  \frac{\log(m)}{m} \theta_m r^m \left(1+O\left(\frac{mw}{n}\right)\right)\nonumber\\
&=
\sum_{m=1}^{b_n}  \frac{\log(m)}{m} \theta_m r^m + O\left( \frac{w}{n}\sum_{m=1}^{b_n}  \log(m)\right)\nonumber\\
&=
\sum_{m=1}^{b_n}  \frac{\log(m)}{m} \theta_m r^m + O\left( \frac{w}{\log(n)} \right).
\label{eq_q_b_at_r}
\end{align}
We now use the asymptotic behavior of $g_\Theta(t)$ at $r$ in \eqref{eq_class_F_r_alpha_near_r} to get
%
%
\begin{align}
& \frac{1}{2 \pi i} 
\int_{\gamma_1\cup\gamma_2\cup\gamma_3}  \widetilde{q}_1(t)\exp\left(g_\Theta(t) \right) \frac{dt}{t^{n+1}}\nonumber\\
=
 &\frac{n^{\vartheta-1}}{2 \pi i r^n} e^K
\int_{\gamma'}  \widetilde{q}_1\left(r+\frac{rw}{n}\right)(-w)^{-\vartheta} e^{-w} (1+O(w/n))\, dw
\nonumber\\
=
 &\frac{n^{\vartheta-1}}{2 \pi i r^n} e^K \left(
\sum_{m=1}^{b_n}  \frac{\log(m)}{m} \theta_m r^m \int_{\gamma'}  (-w)^{-\vartheta} e^{-w} dw + O\left(\log^{-1} (n) \right) \right),
\label{eq:E_tilde_Yn_in_proof_2}
\end{align}
where $\gamma'$ is the bounded curve in Figure~\ref{fig:gamma'}. 
%
%
We have used for the estimate of the reminder that $\Re(e^{-w})$ is decreasing exponentially fast as $\Re(w)\to \infty$.
Furthermore, we can replace with the same observation and a simple contour argument the bounded curve $\gamma'$ with the infinite Hankel contour $\gamma''$ as in Figure~\ref{fig:gamma''}.
Notice that
\begin{align}
\label{eq:classF_integral_Gamma}
 \frac{1}{2\pi i} 
\int_{\gamma''} (-w)^{-\vartheta} e^{-w}\ dw
=
\frac{1}{\Gamma(\vartheta)},
\end{align}
where $\vartheta\in\C$ is arbitrary (details can be found for instance in \cite[Section~B.3]{FlSe09}). 
Combining \eqref{eq:classF_integral_Gamma} with \eqref{eq:E_tilde_Yn_in_proof_2} and Corollary~\ref{cor:haviour_of_hn} completes the proof of the first assertion.
The argument for the second is very similar. One only has to replace \eqref{eq_q_b_at_r} by
\begin{align*}
 \widetilde{e}\left(s, r+\frac{rw}{n}\right)
&=
\sum_{m=1}^{b_n}  \big(e^{s\frac{ \log(m)}{\log(n)}}-1\big)\frac{\theta_m}{m}r^m + O\left( \frac{w}{n} \right).
\end{align*}
\end{proof}

\begin{remark}
\label{rem:why_not_Yn}
Instead of the truncated sequence  $\log \widetilde{Y}_n$ one may consider the generating functions for $\log Y_n$ which are given by
%
\begin{align*}
 \ET{\log Y_n} 
&=   \frac{1}{h_n}\nth{q_1(t)\exp\left(g_\Theta(t) \right)}, \\
\ET{e^{s\log Y_n}} 
&=
\frac{1}{h_n}\nth{
\exp\bigl(e(s,t)+g_\Theta(t) \bigr)}
\end{align*}
with
\begin{align*}
  q_1(t)\coloneqq \left(\sum_{m=1}^\infty \log(m)\frac{\theta_m}{m}t^m\right)
  \quad  \text{ and } \quad
   e(s,t)
=
\sum_{m=1}^{\infty} (e^{s \log(m)}-1)\frac{\theta_m}{m}t^m.
\end{align*}
To use the same contour as in the proof of Lemma~\ref{lem:E_tilde_Yn},  analytic extensions of $q_1(t)$ and $e(s,t)$ to some $\Delta_0$-domain plus the asymptotic behavior at $r$ are required. 
However, for all probabilistic question we consider here, except the precise expected value of $\log O_n$ in Section~\ref{subsection:classF_E_On},
it is enough to know the behavior of the truncated variables $\log \widetilde{Y}_n$ since they are transferable to $\log Y_n$.
\end{remark}
\begin{remark} \label{remark:classF_add_assumption_sum1}
To simplify computations, we will  assume in some cases 
$$\theta_m r^m =\vartheta +O(m^{-\delta})$$
for some $\delta >0$. Then the Euler Summation formula \eqref{eq:intro_euler_maclaurin} yields
 \begin{align}
 \label{eq:classF_add_assumption_sum1}
  &\sum_{m=1}^{b_n} \frac{\log(m)}{m} \theta_m r^m
  =
  \vartheta \sum_{m=1}^{b_n} \frac{\log(m)}{m}  +O(1)
  = 
  \frac{\vartheta}{2} \log^2(b_n) +  O(1), \\
&  \sum_{m=1}^{b_n} \frac{\theta_m}{m}  r^m \one_{\{k|m\}}
  =
  \vartheta \frac{\log(b_n)}{k} + O\left( \frac{\log(k)}{k}\right). \nonumber
 \end{align}
\end{remark}
With this assumption, we get a nice expression for the moment generating function of $\log \widetilde{Y}_n$.
\begin{corollary}
\label{cor:mod_conv_log Yn}
 If $g_\Theta\in\mathcal{F}(r,\vartheta, K)$ and  $\theta_m r^m =\vartheta +O(m^{-\delta})$ for some $\delta >0$, then
 \begin{align*}
 \ET{e^{s\frac{\log \widetilde{Y}_n}{\log(n)}}}
&=
\exp\left(\log(b_n) \left(\frac{e^s}{s}-\frac{1}{s}-1\right) + O\left(\frac{s}{\log(n)}\right)\right)
\left(1+O(n^{-1})  \right).
\end{align*}
%
%
%
\end{corollary}
\begin{proof}
Corollary~\ref{cor:mod_conv_log Yn} follows immediately from Lemma~\ref{lem:E_tilde_Yn} and a simple application of the Euler summation formula \eqref{eq:intro_euler_maclaurin}.
\end{proof}
%

%
%
%
%
%

\vskip 15pt 
\subsection{A local limit theorem for $\log O_n$}\label{subsection:classF_mod_conv_local_limit}
In this section we prove that, given the characteristic function of $\log \widetilde{Y}_n$ in Lemma~\ref{lem:E_tilde_Yn}, the local behavior of the rescaled order of a permutation is well-controlled. 
To this aim, define
\begin{align*}
\mathcal{\widetilde{Y}}_n \coloneqq \frac{\log \widetilde{Y}_n - \frac{\vartheta}{2}\log^2(n)}{\log^{4/3}(n)}.
\end{align*} 
We will show that $\mathcal{\widetilde{Y}}_n$ satisfies the so-called mod-Gaussian convergence; this notion was introduced in 2011 by Jacod et al. \cite{JaKoNi09}. It has interesting applications when typically a sequence of random variables $X_n$ does not converge in distribution, meaning that the sequence of characteristic functions does not converge pointwise to a limit characteristic function, but nevertheless, the characteristic functions decay precisely like those of a suitable Gaussian $G_n$. Specifically, the convergence
\begin{align*}
\mathbb{E}[e^{it G_n}]^{-1} \mathbb{E}[e^{it X_n}] 
\rightarrow \psi(t)
\end{align*}
holds locally uniformly for $t \in \R$, where the limiting function $\psi$ is continuous on $\mathbb{R}$ with $\psi(0)=1$. More generally, mod-$\phi$ convergence with respect to other laws $\phi$ may be defined analogously. In a series of papers \cite{DeKoNi11, FeMeNi13, KoNi09b}, properties and implications of this convergence were studied. 
Here, we will apply Theorem 5 in \cite{DeKoNi11} to show that the mod-Gaussian convergence of the sequence $\mathcal{\widetilde{Y}}_n$ implies a local limit theorem for
\begin{align}\label{eq:classF_local_limit_mathcal_O}
\mathcal{O}_n \coloneqq \frac{\log O_n - \frac{\vartheta}{2}\log^2(n)}{\log^{4/3}(n)}  .
\end{align}  
We will prove 
\begin{theorem}\label{thm:classF_local_limit}
 Suppose that $g_\Theta\in \mathcal{F}(r,\vartheta,K)$ and $\theta_m r^m =\vartheta +O(m^{-\delta})$ for some $\delta>0$.
 For any bounded Borel subset $B \subset \R$ with boundary of Lebesgue measure zero
\begin{align*}
 \lim_{n \rightarrow \infty} \sigma_n \, \PT{\mathcal{O}_n \in B} = \frac{m(B)}{\sqrt{2\pi}},
\end{align*}
where $m(B)$ denotes the Lebesgue measure of $B$  and $\sigma_n = \sqrt{\frac{\vartheta}{3}}\log^{1/6}(n)$.
\end{theorem}
To prove this, let us first show that $\mathcal{\widetilde{Y}}_n$ is indeed mod-Gaussian convergent in Lemma~\ref{lem:classF_mod-conv_Y_n}. Subsequently, we present in Lemma~\ref{lem:classF_local_limit_Y_n} that $\mathcal{\widetilde{Y}}_n$ satisfies the required local behavior. Finally, the result has to be transferred to $\mathcal{O}_n$.
\begin{lemma}\label{lem:classF_mod-conv_Y_n}
Under the assumptions of Theorem~\ref{thm:classF_local_limit}, the sequence $\mathcal{\widetilde{Y}}_n$ is mod-$\mathcal{N}(0,\sigma_n^2)$ convergent with $\sigma_n^2 = \frac{\vartheta}{3}\log^{1/3}(n)$ and limiting function given by $\psi(x)=e^{x^3\vartheta / 18}$. 
\end{lemma} 
\begin{proof}
Take the generating function in Lemma~\ref{lem:E_tilde_Yn} and expand the exponential term to get
\begin{align*}
\ET{e^{s\frac{\log \widetilde{Y}_n}{\log n}}} 
& =
\exp\Big(s \frac{\vartheta}{2} \big(\log(n) + O(\log\log(n)\big) + 
\frac{s^2}{2} \frac{\vartheta}{3} \big(\log(n) + O(\log\log(n)\big) \\
& + \frac{s^3}{3!} \frac{\vartheta}{4} \big(\log(n) + O(\log\log(n)\big)  
+ O(s^4 \log(n))\Big)
\left(1+O(n^{-1})  \right);
\end{align*}
we used \eqref{eq:classF_add_assumption_sum1} and similar estimates for the higher order terms. Since $s \in \C$ we may write $s = it$ with $t \in \R$ and get
%
%
\begin{align*}
\ET{e^{it\frac{\log \widetilde{Y}_n}{\log^{4/3} n}}}
&=
\exp\Big( it \frac{\vartheta}{2} \log^{\frac{2}{3}}(n) - 
\frac{t^2}{2} \frac{\vartheta}{3} \log^{\frac{1}{3}}(n) 
+ \frac{it^3}{3!} \frac{\vartheta}{4}  + O\Big(\frac{\log\log(n)}{\log^{1/3}(n)}\Big)\Big)
\end{align*}
and this gives the result.
\end{proof}

 As a direct consequence, we get a local limit theorem for $\mathcal{\widetilde{Y}}_n$.  
\begin{lemma}\label{lem:classF_local_limit_Y_n}
Under the assumptions of Theorem~\ref{thm:classF_local_limit} the following holds for any bounded Borel subset $B \subset \R$ with boundary of Lebesgue measure zero:
\begin{align*}
 \lim_{n \rightarrow \infty} \sigma_n \, \PT{\mathcal{\widetilde{Y}}_n \in B} = \frac{m(B)}{\sqrt{2\pi}},
\end{align*}
where $m(B)$ denotes the Lebesgue measure of $B$  and $\sigma_n$ is defined as in Lemma~\ref{lem:classF_mod-conv_Y_n}.
\end{lemma}

\begin{proof}
Apply Theorem 5 in \cite{DeKoNi11}  with $\varphi(t)=e^{-t^2/2}$ and $A_nt=\sigma_nt$. We need to verify that  condition H3 holds, that is we have to show the uniform integrability of the sequence
\begin{align*}
f_{nk}\coloneqq\mathbb{E}_{\Theta}\big[e^{it\frac{\mathcal{\widetilde{Y}}_n}{\sigma_n}}\big] \one_{|t \sigma_n^{-1}| \leq k}
\end{align*}
for all $k \geq 0$. Set $\overline{t}\coloneqq t/\sigma_n$ and recall that Lemma~\ref{lem:classF_mod-conv_Y_n} implies
\begin{align*}
\mathbb{E}_{\Theta}\big[e^{i \overline{t} \mathcal{\widetilde{Y}}_n}\big] 
&= \exp \Big(- \overline{t}^2 \sigma_n^2 + \frac{\vartheta}{18} i\overline{t}^3 + O(\overline{t}^4 \frac{\log\log(n)}{ \log^{1/3}(n)})\Big)\\
&= \exp \Big(- t^2 + \frac{\vartheta}{18}\frac{ it^3}{\sigma_n^3} + O(t^4 \frac{\log\log(n)}{\sigma_n^4 \log^{1/3}(n)})\Big).
\end{align*}
Thus
\begin{align*}
 \big|\mathbb{E}_{\Theta}\big[e^{i \overline{t} \mathcal{\widetilde{Y}}_n}\big] \big|
= \exp \big(- t^2 + o(1)\big)
\end{align*}
which implies the uniform integrability.
\end{proof}

  \begin{proof}[Proof of Theorem~\ref{thm:classF_local_limit}]
 It remains to transfer the result from $\mathcal{\widetilde{Y}}_n$ to
  $$\mathcal{\widetilde{O}}_n \coloneqq \frac{\log \widetilde{O}_n - \frac{\vartheta}{2}\log^2(n)}{\log^{4/3}(n)}$$
and subsequently to $\mathcal{O}_n$ defined as in Theorem~\ref{thm:classF_local_limit}. To this aim, notice that for every $\epsilon > 0$ there exist Jordan-measurable sets (meaning that they are bounded with boundary of Lebesgue measure zero) $B_{\epsilon} \subset B \subset B^{\epsilon}$ such that
 $$m(B^{\epsilon} \setminus B) \leq \epsilon 
 \quad \text{ and } \quad
 m(B \setminus B_{\epsilon}) \leq \epsilon .$$
 To see this, notice that $\partial B$ is bounded (since $B$ is bounded) and that it is also closed (complement of the interior and the exterior, both open sets), thus $\partial B$ is compact. Cover $\partial B$ with open rectangles whose total volume does not exceed $\epsilon$. Since $\partial B$ is compact, $U$ can be chosen to be a finite union of open rectangles. Then define
  $$B_{\epsilon} \coloneqq B \setminus U
 \quad \text{ and } \quad
 B^{\epsilon}\coloneqq B \cup U $$
 to get the required sets (they are indeed Jordan-measurable since $\partial (B \setminus U) \subset \partial B \cup \partial U$ and $\partial (B \cup U) \subset \partial B \cup \partial U$). This gives
\begin{align*}
\PT{\mathcal{\widetilde{O}}_n \in B} 
\leq \PT{\mathcal{\widetilde{Y}}_n \in B^{\epsilon}} 
+ O \Big( \PT{ \log \widetilde{Y}_n - \log \widetilde{O}_n \geq \epsilon \log^{4/3}(n) } \Big)
\end{align*} 
 and 
 \begin{align*}
\PT{\mathcal{\widetilde{O}}_n \in B} 
\geq \PT{\mathcal{\widetilde{Y}}_n \in B^{\epsilon}} 
+ O \Big( \PT{ \log \widetilde{Y}_n - \log \widetilde{O}_n \geq \epsilon \log^{4/3}(n) } \Big).
\end{align*} 
Thus, we have to show 
 \begin{align}\label{eq:classF_local_limit_theorem_tildeY_n_zu_Yn}
\sigma_n\PT{ \log \widetilde{Y}_n - \log \widetilde{O}_n \geq \epsilon \log^{4/3}(n) }\rightarrow 0.
\end{align} 
This is true since 
\begin{align*}
\PT{ \log \widetilde{Y}_n - \log \widetilde{O}_n \geq \epsilon \log^{4/3}(n) }
\leq \PT{ \log Y_n - \log O_n \geq \epsilon \log^{4/3}(n) }
\end{align*} 
and then \eqref{eq:classF_expectation_Delta_n} and Markov's inequality yield the required asymptotic. Now \eqref{eq:classF_local_limit_theorem_tildeY_n_zu_Yn} implies
\begin{align*}
\lim_{n \rightarrow \infty} \sigma_n \PT{\mathcal{\widetilde{O}}_n \in B}  
\leq \lim_{n \rightarrow \infty} \sigma_n \PT{\mathcal{\widetilde{Y}}_n \in B^{\epsilon}} 
=  \frac{m(B^{\epsilon})}{\sqrt{2\pi}}
\leq \frac{m(B)+\epsilon}{\sqrt{2\pi}}.
\end{align*}
With the same argument for the reversed inequality, we get that for all $\epsilon >0$, 
\begin{align*}
\frac{m(B)-\epsilon}{\sqrt{2\pi}}
\leq \lim_{n \rightarrow \infty} \sigma_n \PT{\mathcal{\widetilde{O}}_n \in B}  
\leq \frac{m(B)+\epsilon}{\sqrt{2\pi}}.
\end{align*}
Let $\epsilon$ tend to zero to obtain 
\begin{align*}
 \lim_{n \rightarrow \infty} \sigma_n \, \PT{\mathcal{\widetilde{O}}_n \in B} 
 = \frac{m(B)}{\sqrt{2\pi}}.
\end{align*}
With the same argument, the result is transferred from $\mathcal{\widetilde{O}}_n$ to $\mathcal{O}_n$, assuming that
 \begin{align*}
\sigma_n\PT{ \log O_n - \log \widetilde{O}_n \geq \epsilon \log^{4/3}(n) }\rightarrow 0
\end{align*} 
is satisfied. To see this, notice that
 \begin{align*}
\PT{ \log O_n - \log \widetilde{O}_n \geq \epsilon \log^{4/3}(n) } 
\leq \PT{ \log Y_n - \log \widetilde{Y}_n \geq \epsilon \log^{4/3}(n) }
\end{align*}
holds as well as
 \begin{align*}
\ET{ \log Y_n - \log \widetilde{Y}_n } = O\big(\log(n)\log\log(n)\big).
\end{align*} 
\end{proof}

\vskip 15pt 
\subsection{Large deviations estimates for $\log O_n$}\label{subsection:classF_large_dev}

This section is devoted to two large deviations estimates for $\log O_n$. To our knowledge, these results are new even for the uniform measure. The first estimate is established by a classical large deviations approach. We will show in Theorem~\ref{thm:classF_large_dev_logO_n} that for any Borel set $B$
\begin{align}\label{eq:classF_large_dev_classical}
\limsup_{n \rightarrow \infty} \frac{1}{\log(n)} \log\mathbb{P}_{\Theta}\bigg(\frac{\log O_n}{\log^{2}(n)} \in B\bigg) 
= -\inf_{x\in B} F(x)
\end{align}
where
\begin{align*}
F(x) \coloneqq \sup_{t \in \R}[t x - \chi(t)]
\end{align*}
is the so-called Fenchel-Legendre transform of $\chi(t)\coloneqq \frac{e^t -1-t}{t}$.
This result was stated by O'Connell \cite{OCo96} for the uniform measure. However, we believe his proof of Lemma 2 is incorrect and we don't see an easy way to fix it.
Here, we give a detailed proof based on an extra moment condition and even present a refined result, namely a precise large deviations estimate; see Theorem~\ref{thm:classF_mod_large-dev1}. 

\textbf{Moment condition} 
Assume that $g_\Theta$ belongs to $\mathcal{F}(\rho,\vartheta, K)$ and assume  $\theta_m r^m =\vartheta +\mathcal{O}(m^{-\delta})$ for some $\delta >0$. Define
\begin{align*}
\Delta_{n,\beta(n)} \coloneqq \sum_{k = \beta(n)}^n \Lambda(k)(\widetilde{D}_{nk}-\widetilde{D}_{nk}^*),
\end{align*}
where $\beta(n)=\exp(\log^x(n))$ for some $x<1$. Then the moment condition is satisfied if 
there exists an $n_0\in\N$ and a sequence $(k_m)_{m\in\N}$ 
such that for all $n\geq n_0$ the following holds:
\begin{align}\label{eq:classF_large_dev_moments}
\left|\ET{(\Delta_{n,\beta(n)})^m}\right| \leq k_m(\log(n)\log\log(n))^m\big)
\end{align}
with $k_m =O(e^{\alpha \, m})$ for some $\alpha>0$ with $\alpha$ independent of $n$ and $m$.
%
%
%
%
\begin{remark}
We are strongly convinced that the moment condition is satisfied under the above assumptions, however we are so far not able to prove it. 
The condition is clearly satisfied for $m=1$ and for $m=2$ 
 and the computations for these cases can be found for instance in the Appendix in \cite{St15}. Furthermore, we have been able to show that 
 \begin{align*}
\ET{(\Delta_{n,\beta(n)})^m} = O_m\big((\log(n)\log\log(n))^m\big),
\end{align*}
but we couldn't very the upper bound for $k_m$. However, this computations are very technical and we thus don't state them here.
\end{remark}
With the moment generating function of $\log \widetilde{Y}_n / \log(n)$ stated in Corollary~\ref{cor:mod_conv_log Yn} at hand, a simple application of the G\"artner-Ellis Theorem yields an estimate as in (\ref{eq:classF_large_dev_classical}) for $\log \widetilde{Y}_n$. Then, using the moment condition \eqref{eq:classF_large_dev_moments}, we show by exponential equivalence that this estimate can be transferred to $\log \widetilde{O}_n$ and then to $\log O_n$. More precisely, we will prove the following

\begin{theorem}\label{thm:classF_large_dev_logO_n}
Let $g_\Theta$ belong to $\mathcal{F}(r,\vartheta, K)$, $\theta_m r^m =\vartheta +O(m^{-\delta})$ for some $\delta >0$ and assume that the moment condition \eqref{eq:classF_large_dev_moments} holds. Then the sequence $\log O_n/ \log^2(n)$ satisfies a large deviations principle with rate $\log(n)$ and rate function given by the Fenchel-Legendre transform of $\chi(t)\coloneqq \frac{e^t -1-t}{t}$.
\end{theorem}
\begin{proof}
Let us first check that $\log \widetilde{Y}_n/\log^2(n)$ satisfies the required large deviations principle. By the G\"artner-Ellis Theorem, it suffices to check
\begin{align*}
\lim_{n \rightarrow \infty} \frac{1}{\log(n)} \log\ET{\exp\Big(t \frac{\log \widetilde{Y}_n}{\log(n)}\Big)} = \chi(t)
 \end{align*}
and this follows immediately from Corollary~\ref{cor:mod_conv_log Yn}. 
Proving exponential equivalence, Lemma~\ref{lem:classF_exp_equiv_classical} transfers this result from $\log \widetilde{Y}_n$ to $\log \widetilde{O}_n$ and then to $\log O_n$.
\end{proof}

\begin{lemma}\label{lem:classF_exp_equiv_classical}
Under the assumptions of Theorem~\ref{thm:classF_large_dev_logO_n} the following holds for any $c>0$:
\begin{enumerate}
\item[\textup{(1)}] 
$
\limsup_{n \rightarrow \infty} \frac{1}{\log(n)} 
\log \PT{\log \widetilde{Y}_n - \log \widetilde{O}_n >c \log^2(n)} = -\infty, \\
$
\item[\textup{(2)}] 
$
\limsup_{n \rightarrow \infty} \frac{1}{\log(n)} 
\log \PT{\log O_n - \log \widetilde{O}_n >c \log^2(n)} = -\infty.
$
\end{enumerate}
\end{lemma}
\begin{proof}
We will prove stronger versions of $(1)$ and $(2)$ in Lemma~\ref{lem:classF_exp_equiv_logYn_logOn} and Lemma~\ref{lem:classF_exp_equiv_log_tildeOn_logOn} below.
\end{proof}

The result of Theorem~\ref{thm:classF_large_dev_logO_n} can be even refined:
\begin{theorem}\label{thm:classF_mod_large-dev1}
Let $\mathcal{O}_n$ be as in \eqref{eq:classF_local_limit_mathcal_O} and $\sigma_n^2 = \frac{\vartheta}{3}\log^{1/3}(n)$. Then, under the assumptions of Theorem~\ref{thm:classF_large_dev_logO_n}, for any $x>0$ the following holds:
\begin{align*}
 \PT{\mathcal{O}_n \geq x \, \sigma_n^2} = \frac{\exp(-\sigma_n^2 \frac{x^2}{2} + \frac{x^3\vartheta}{18})}{\sqrt{2\pi \sigma_n^2 x^2}} \,  (1+o(1)).
\end{align*}
\end{theorem}
To prove this result, we proceed as follows: from the mod-Gaussian convergence of $\mathcal{\widetilde{Y}}_n$ stated in Lemma~\ref{lem:classF_mod-conv_Y_n} we deduce a precise large deviations estimate for $\mathcal{\widetilde{Y}}_n$. Then, using the moment condition \eqref{eq:classF_large_dev_moments} we prove exponential equivalence similar to Lemma~\ref{lem:classF_exp_equiv_classical} to transfer the estimate to $\mathcal{O}_n$.
\begin{proof}[Proof of Theorem~\ref{thm:classF_mod_large-dev1}]
 First, combine Lemma~\ref{lem:classF_mod-conv_Y_n} with Theorem 3.2 in \cite{FeMeNi13} to get the same precise deviations estimate for $\mathcal{\widetilde{Y}}_n$ (take $t_n = \sigma_n^2$, $F(x)= x^2/2 = \eta(x)$ and $\varphi(x)$ as in Lemma~\ref{lem:classF_mod-conv_Y_n}). Lemma~\ref{lem:classF_exp_equiv_logYn_logOn} below transfers the result to $\mathcal{\widetilde{O}}_n$ and subsequently Lemma~\ref{lem:classF_exp_equiv_log_tildeOn_logOn} transfers the result to $\mathcal{O}_n$.
\end{proof}

\begin{lemma}\label{lem:classF_exp_equiv_log_tildeOn_logOn}
Under the assumptions of Theorem~\ref{thm:classF_mod_large-dev1} the following holds for any $c>0$:
\begin{align*}
\lim_{n \rightarrow \infty} \frac{1}{\sigma_n^2} 
\log \PT{\log O_n - \log \widetilde{O}_n >c \log^{4/3}(n)} = -\infty.
\end{align*}
\end{lemma}

\begin{proof}
We have 
\begin{align*}
\log O_n - \log \widetilde{O}_n \leq \log Y_n - \log \widetilde{Y}_n
\end{align*}
and thus the assertion is proved if we can show
\begin{align*}
\lim_{n \rightarrow \infty} \frac{1}{\log^{1/3}(n)} 
\log \PT{\log Y_n  - \log \widetilde{Y}_n  >c \log^{4/3}(n)} = -\infty.
\end{align*}
Define
\begin{align*}
D(n,b) \coloneqq \log Y_n  - \log \widetilde{Y}_n  = \sum_{m=b_n +1}^n \log(m) \,C_m 
\end{align*}
and notice that
\begin{align*}
\PT{\log Y_n - \log \widetilde{Y}_n >c \log^{4/3}(n)} 
\leq \PT{ T(b_n,n)>c \log^{1/3}(n)} 
\end{align*}
where
\begin{align*}
T(b_n, n) = \sum_{m=b_n +1}^n C_m.
\end{align*}
 Thus it suffices to show 
\begin{align*}
\lim_{n \rightarrow \infty} \frac{1}{\log^{1/3}(n)}
\log \PT{T(b_n, n)>c \log^{1/3}(n)} = -\infty.
\end{align*}
With Markov's inequality we get 
\begin{align}\label{eq:algrowth_LDP_tildeYn_zuY_n}
  \frac{1}{\log^{1/3}(n)}\log\mathbb{P}_{\Theta}\Big(e^{sT(b_n,n)}\geq e^{s c \log^{1/3}(n)}\Big) 
\leq  -sc  + \frac{\log \ET{e^{sT(b_n, n)}}}{\log^{1/3}(n)}.
\end{align}
The generating function of $T(b_n, n)$ is given by
\begin{align*}
 \log \ET{e^{sT(b_n, n)}} = \vartheta (e^s-1) \log(n) + (K - L_{b_n}(r))(e^s-1) + o(1)
\end{align*}
where
\begin{align*}
 L_{b_n}(r) 
 = \sum_{m=1}^{b_n} \frac{\theta_m}{m} r^m 
  = \vartheta \sum_{m=1}^{b_n} \frac{1}{m} + O(1)
  = \vartheta \log(b_n) + O(1),
\end{align*}
see \cite[Theorem 4.3]{NiStZe13}  with $A_n = \{1, ..., b_n \}$. Thus
\begin{align*}
\eqref{eq:algrowth_LDP_tildeYn_zuY_n} \leq -sc + O \Big( \frac{e^s \log\log(n)}{\log^{1/3}(n)}\Big)
\end{align*}
and choose $s = \log\log\log(n)$ to get the result.
\end{proof}

\begin{lemma}\label{lem:classF_exp_equiv_logYn_logOn}
Under the assumptions of Theorem~\ref{thm:classF_mod_large-dev1} the following holds for any $c>0$:
\begin{align*}
\lim_{n \rightarrow \infty} \frac{1}{\sigma_n^2} 
\log \PT{\log \widetilde{Y}_n - \log \widetilde{O}_n >c \log^{4/3}(n)} = -\infty.
\end{align*}
\end{lemma}
\begin{proof}
 We use \eqref{eq:log_tildeYn_von_mangold} and  \eqref{eq:log_tildeOn_von_mangold} an get  for $\beta(n)\in\N$ (determined later)
 \begin{align*}
  \log \widetilde{Y}_n - \log \widetilde{O}_n
  &=
  \sum_{m=1}^{\beta(n)} \Lambda(k)(\widetilde{D}_{nk}-\widetilde{D}_{nk}^*)
  +
  \sum_{m=\beta(n)+1}^{b_n} \Lambda(k)(\widetilde{D}_{nk}-\widetilde{D}_{nk}^*)\\
  &\leq
  \sum_{m=1}^{\beta(n)} \Lambda(k)\widetilde{D}_{nk}
  +
  \sum_{m=\beta(n)+1}^{b_n} \Lambda(k)(\widetilde{D}_{nk}-\widetilde{D}_{nk}^*)\\
  &\leq
  \sum_{m=1}^{\beta(n)}  \log(m) \, C_m
  +
  \sum_{m=\beta(n)+1}^{b_n} \Lambda(k)(\widetilde{D}_{nk}-\widetilde{D}_{nk}^*)
 \end{align*}
We thus have
\begin{align*}
 \PT{\log \widetilde{Y}_n - \log \widetilde{O}_n >c \log^{4/3}(n)}
 &\leq  
 \PT{\sum_{m=1}^{\beta(n)} \log(m) \, C_m > \frac{c}{2}\log^{4/3}(n)}\\
 &+
 \PT{\sum_{k=\beta(n)}^n \Lambda(k)(\widetilde{D}_{nk}-\widetilde{D}_{nk}^*) > \frac{c}{2}\log^{4/3}(n)}.
\end{align*}

Notice that for any sequences $(a_n)_{n\in \N}$ and $(b_n)_{n\in \N}$ with  $a_n,b_n\in (0, \infty)$ and any $g(n) \rightarrow \infty$
\begin{align*}
\limsup_{n\rightarrow \infty} \frac{\log(a_n + b_n)}{g(n)}  
= \max \Big\{ \limsup_{n\rightarrow \infty} \frac{\log(a_n)}{g(n)} ,\limsup_{n\rightarrow \infty} \frac{\log(b_n) }{g(n)} \Big\}.
\end{align*}
We want to find the biggest $\beta(n)$ such that 
\begin{align}\label{eq:classF_LDP_transfer_1}
\frac{1}{\log^{1/3}(n)} \log\PT{\sum_{m=1}^{\beta(n)} \log(m) \, C_m > \frac{c}{2}\log^{4/3}(n)}
\rightarrow - \infty
\end{align}
is satisfied. Subsequently, by means of the moment condition \eqref{eq:classF_large_dev_moments} we show 
\begin{align}\label{eq:classF_LDP_transfer_2}
\frac{1}{\log^{1/3}(n)} \log\PT{\sum_{k=\beta(n)}^n \Lambda(k)(\widetilde{D}_{nk}-\widetilde{D}_{nk}^*) > \frac{c}{2}\log^{4/3}(n)}
\rightarrow - \infty.
\end{align}
 We start with \eqref{eq:classF_LDP_transfer_1}.
For any $s>0$, Markov's inequality yields
\begin{align}
&\frac{1}{\log^{1/3}(n)} \log\PT{\sum_{m=1}^{\beta(n)} \log(m) \, C_m > \frac{c}{2}\log^{4/3}(n)} \nonumber \\
= \,& \frac{1}{\log^{1/3}(n)} \log\PT{\exp\Big(s\sum_{m=1}^{\beta(n)} \log(m) \, C_m\Big )> \exp\Big(\frac{sc}{2}\log^{4/3}(n)\Big)} \nonumber\\
\leq \, & - \frac{sc}{2}\log(n) + \frac{1}{\log^{1/3}(n)} \log\ET{\exp\Big(s\sum_{m=1}^{\beta(n)} \log(m) \, C_m\Big )}. \label{eq:classF_LD_1}
\end{align}
The asymptotic behaviour of the moment generating function in \eqref{eq:classF_LD_1} can be computed in exactly the same way as the moment generating function 
of $\log \widetilde{Y}_n$ in Lemma~\ref{lem:E_tilde_Yn}. Indeed, only minor modifications are required and we thus omit the computation.
This then gives for $s$, $\beta(n)$ with $s\log \beta(n) =o(1)$ 
\begin{align*}
\ET{\exp\Big(s\sum_{m=1}^{\beta(n)} \log(m) \, C_m\Big )} 
= \exp\left( \sum_{m=1}^{\beta(n)} (e^{s\log m}-1) \frac{\theta_m}{m}r^m \right)(1+o(1)).
\end{align*}
Using the assumption $\theta_m r^m = \vartheta + O(m^{-\delta})$ then gives
\begin{align*}
 \frac{1}{\log^{1/3}(n)} \log\ET{e^{s\sum_{m=1}^{\beta(n)} \log(m) \, C_m}}
 %
%
& =  \frac{\vartheta}{\log^{1/3}(n)} \left(\sum_{m=1}^{\beta(n)} \frac{e^{s\log(m)}-1}{m}\right) + o(1). 
\end{align*}
%
%
%
%

Now set $s\coloneqq \log\log(n) / \log(n)$, then for all $\beta(n) = \exp(o(\log n/\log\log n)) $, 
\begin{align*}
 \frac{\vartheta}{\log^{1/3}(n)} \sum_{m=1}^{\beta(n)} \frac{e^{s\log(m)}-1}{m}
& = O\bigg( \frac{\log\log(n)}{\log^{4/3}(n)} \sum_{m=1}^{\beta(n)} \frac{\log(m)}{m}\bigg)\\
& = O\bigg( \frac{\log\log(n)\log^2(\beta(n))}{\log^{4/3}(n)} \bigg).
\end{align*}
Thus, set $\beta(n) \coloneqq \exp(\sqrt{\log(n)})$ to obtain
\begin{align*}
\eqref{eq:classF_LD_1}  
= - \frac{c \log\log(n)}{2} 
+ O\bigg( \frac{\log\log(n)}{\log^{1/3}(n)} \bigg)
\end{align*}
and therefore assertion \eqref{eq:classF_LDP_transfer_1} is proved. So let us consider \eqref{eq:classF_LDP_transfer_2}. Again, for $s>0$, and with the notation from the moment condition \eqref{eq:classF_large_dev_moments},
\begin{align*}
& \frac{1}{\log^{1/3}(n)} \log\PT{\Delta_{n,\beta(n)} > \frac{c}{2}\log^{4/3}(n)} \\
= & \, \frac{1}{\log^{1/3}(n)} \log\PT{e^{s \Delta_{n,\beta(n)}} > e^{\frac{sc}{2}\log^{4/3}(n)}} \\
\leq & \, -\frac{sc}{2}\log(n) + \frac{1}{\log^{1/3}(n)} \log\ET{e^{s \Delta_{n,\beta(n)}}} .
\end{align*}
Thus, we set again $s\coloneqq \log\log(n) / \log(n)$. Define the event  
$$A \coloneqq \Big\{\Delta_{n,\beta(n)} \leq \frac{\log^{4/3}(n)}{\log\log(n)}\Big\}. $$
Then for  $s= \log\log(n) / \log(n)$
\begin{align*}
\ET{e^{s \Delta_{n,\beta(n)}}} 
&= \ET{e^{s \Delta_{n,\beta(n)}}\one_{\{A\}}} 
+ \ET{e^{s \Delta_{n,\beta(n)}}\one_{\{A^c\}}} \\
&\leq e^{\log^{1/3}(n)}
+ \ET{e^{s \Delta_{n,\beta(n)}}\one_{\{A^c\}}} .
\end{align*} 
We will show that 
\begin{align}\label{eq:classF_large_dev_O1}
 \frac{1}{\log^{1/3}(n)} \log\ET{e^{s \Delta_{n,\beta(n)}}\one_{\{A^c\}}}  = O(1)
\end{align}
 holds. Cauchy's inequality yields
\begin{align*}
\ET{e^{s \Delta_{n,\beta(n)}}\one_{\{A^c\}}} 
&\leq \PT{A^c}^2\ET{e^{2s \Delta_{n,\beta(n)}}}  \\
&\leq 
\PT{A^c} \sum_{m=0}^{g(n)} \frac{\ET{(2s\Delta_{n,\beta(n)})^m}}{m!}\\
&+\sum_{m=g(n)+1}^{\infty} \frac{\ET{(2s\Delta_{n,\beta(n)})^m}}{m!},
\end{align*}
where $g(n)$ is a function to be determined in a moment. By the moment condition \eqref{eq:classF_large_dev_moments} and by Stirling's formula we have for $s= \log\log(n) / \log(n)$
\begin{align*}
&\sum_{m=g(n)+1}^{\infty} \frac{\ET{(2s\Delta_{n,\beta(n)})^m}}{m!}\leq 
\sum_{m=g(n)+1}^{\infty} \frac{k_m\, (2s)^m \bigl( \log(n) \log\log(n) \bigr)^m}{m!} 
\\
&\leq
\sum_{m=g(n)+1}^{\infty} \frac{k_m\, 2^m \bigl( \log\log(n) \bigr)^{2m}}{m!} 
=
O\left( \sum_{m=g(n)+1}^{\infty} \frac{  \bigl(2 e^{\alpha}\,  (\log\log(n))^2 \bigr)^m}{m!}   \right)
\\
%
&=  O \bigg( \sum_{m=g(n)+1}^{\infty} \exp\Big(m\log\bigl(2e^{\alpha}\,(\log\log(n))^2\bigr)-m\log(m)\Big)\bigg).
\end{align*}
%
%
Consequently, for $g(n) = (\log\log(n))^3$, this sum satisfies \eqref{eq:classF_large_dev_O1}. On the other hand, by Markov's inequality 
$$\PT{A^c} 
= \PT{\Delta_{n,\beta(n)} > \frac{\log^{4/3}(n)}{\log\log(n)}}
\leq \frac{\log\log(n)}{\log^{4/3}(n)}\ET{\Delta_{n,\beta(n)}} .
$$
Notice that 
$$\widetilde{D}_{nk} - \widetilde{D}^*_{nk}
\leq D_{nk} - D^*_{nk}
\leq D_{nk}(D_{nk}-1)$$ 
and recall that $\beta(n) = \exp(\sqrt{\log(n)})$. Furthermore, recall \eqref{eq:intro_Mangoldt_sum3} and Proposition~\ref{prop:bound_Dnk}. Then
\begin{align*}
\ET{\Delta_{n,\beta(n)}} 
&= \sum_{k = \beta(n)}^n \Lambda(k)\ET{D_{nk}(D_{nk}-1)} 
= O\bigg(\log^2(n)\sum_{k = \beta(n)}^n \frac{\Lambda(k)}{k^2} \bigg)\\
&= O\bigg(\frac{\log^2(n)}{\beta(n)}\bigg)
= O\bigg(\log^2(n) \, e^{-\sqrt{\log(n)}}\bigg).
\end{align*}
This implies
\begin{align*}
\PT{A^c} = O\bigg(\log^{2/3}(n)\log\log(n) \, e^{-\sqrt{\log(n)}}\bigg).
\end{align*}
We thus get with the moment condition \eqref{eq:classF_large_dev_moments} and $s=\log\log(n)/\log(n)$
\begin{align*}
&\PT{A^c} \sum_{m=0}^{(\log\log(n))^3} \frac{\ET{(2s\Delta_{n,\beta(n)})^m}}{m!}
\leq
\PT{A^c} \sum_{m=0}^{(\log\log(n))^3} \frac{k_m\,2^m \bigl(\log\log(n) \bigr)^{2m}}{m!} 
\\
&=  O \bigg( \PT{A^c} \sum_{m=0}^{(\log\log(n))^3} \frac{(2 e^{\alpha})^m \bigl(\log\log(n) \bigr)^{2m}}{m!} \bigg)\\
&=  O \bigg( \PT{A^c} (\log\log(n))^3 \cdot \left[(2 e^{\alpha})^m \bigl(\log\log(n) \bigr)^{2m}\right]\Big|_{m=\log\log(n)}\bigg)\\
%
%
& = O \bigg(\PT{A^c} (\log\log(n))^3  \exp\Big( 2(\log\log(n))^3\log\log\log(n)\Big)  \bigg)\\
& = O \bigg(\log^{2/3}(n)(\log\log(n))^4 \exp\Big(-\sqrt{\log(n)} + 2(\log\log(n))^3\log\log\log(n)\Big)  \bigg).
\end{align*}
Altogether, we proved \eqref{eq:classF_large_dev_O1} and thus \eqref{eq:classF_LDP_transfer_2} holds. The proof is complete.
%
%
%
\end{proof}
\vskip 15pt 
\subsection{Expected value of the logarithm of a truncated order}
\label{sec:E_On_tilde}

Recall the definition of the truncated order $\widetilde{O}_n$ in \eqref{eq:def_On_tilde_Yn_tilde} .
We will compute a precise asymptotic expansion for $\mathbb{E}_{\Theta}[\log\widetilde{O}_n]$ .

%
\begin{theorem}\label{thm:classF_E_logO_n_tilde}
Suppose that $g_\Theta\in \mathcal{F}(r,\vartheta,K)$. Then
\begin{align}
 \ET{\log \widetilde{O}_n}
& = 
\sum_{m=1}^{b_n} \frac{\log(m)}{m} \theta_m r^m
-
\sum_{k=1}^{\log^2(n)} \Lambda(k)\exp\left(-\sum_{m=1}^{b_n}\frac{\theta_m}{m}r^m\one_{\{k|m\}}\right)\nonumber\\
&-
\sum_{k=1}^{\log^2(n)} \Lambda(k) \left(\sum_{m=1}^{b_n}\frac{\theta_m}{m}r^m\one_{\{k|m\}}-1\right)
 +
 O(1).
\label{eq:E_tilde_On_exact}
\end{align}
\end{theorem}
Before we prove this theorem, we point out the following direct consequence. 
\begin{corollary}
\label{cor:classF_E_logO_n_tilde_with_RH}
Suppose that $g_\Theta\in \mathcal{F}(r,\vartheta,K)$ and $\theta_m r^m =\vartheta +O(m^{-\delta})$ for some $\delta>0$.
Then
\begin{align*}
\ET{\log \widetilde{O}_n}
&=
\frac{\vartheta}{2}\log^2(b_n) +
\vartheta \log(b_n)\bigl( \log(\vartheta \log(b_n))  -1\bigr)\nonumber\\
&+
\sum_\rho \Gamma(-\rho) (\vartheta\log(b_n)\bigr)^\rho
+
O\left((\log\log(n))^3 \right),
\end{align*}
where $\sum_\rho$ indicates the sum over the non-trivial zeros $\rho$ of Riemann zeta function.
\end{corollary}

 Assuming the Riemann hypothesis to be true, that is all the non-trivial zeros of the zeta function have the form $\varrho = 1/2 + it$, any sum $\sum_\varrho x^{\varrho}$ with $x \geq 0$ can be estimated as $\mathcal{O}(\sqrt{x})$. This leads to the implication $(1) \Rightarrow (2)$ in the following Corollary. Moreover, similar as for the Chebychev function \eqref{eq:Chebychev_asymp_RH}, we notice that the reverse implication is also true: if there would exist a zero of the zeta function of the form $\varrho = 1/2 + \delta +it$ with $\delta>0$, then we can deduce a contradiction for $\epsilon = \delta/2$. For more details we refer to the proof of \eqref{eq:Chebychev_asymp_RH} in \cite[Section II.4, Corollary 3.1]{Te95}.
\begin{corollary}
\label{cor:Equiv_RH_tilde}
Suppose that $g_\Theta\in \mathcal{F}(r,\vartheta,K)$ and $\frac{\theta_m}{m}r^m =\vartheta +O(m^{-\delta})$ for some $\delta>0$.
Then the following statements are equivalent 
\begin{enumerate}
 \item[\textup{(1)}]  The Riemann hypothesis is true.
 \item[\textup{(2)}]   We have for all $\epsilon>0$ 
\begin{align*}
\ET{\log \widetilde{O}_n}
= 
 \frac{\vartheta}{2}\log^2(b_n) +
\vartheta \log(b_n)\bigl(\log(\vartheta \log(b_n)) -1 \bigr)
+
O\left((\log(b_n))^{1/2+\epsilon} \right).
\end{align*}

\end{enumerate}
\end{corollary}

Let us now deduce Corollary~\ref{cor:classF_E_logO_n_tilde_with_RH} from Theorem~\ref{thm:classF_E_logO_n_tilde}.
\begin{proof}[Proof of Corollary~\ref{cor:classF_E_logO_n_tilde_with_RH}]

 Recall the estimates in Remark~\ref{remark:classF_add_assumption_sum1}. Then
 \begin{align*}
  \ET{\log \widetilde{O}_n}
  &= 
   \frac{\vartheta}{2}\log^2(b_n) 
  -
\sum_{k=1}^{\log^2(n)}  \Lambda(k)\left(e^{-\vartheta \frac{\log(b_n)}{k} } -1+ \vartheta \frac{\log(b_n)}{k}\right) \\
&+ O\left( \sum_{k=1}^{\log^2(n)}\Lambda(k)\frac{\log(k)}{k}\right).
 \end{align*}
Since $\Lambda(k) \leq \log(k)$, the sum over the error term is of order
\begin{align*}
&\sum_{k=1}^{\log^2(n)}  \Lambda(k) O\left(\frac{\log(k)}{k}\right)
=
O\left(\sum_{k=1}^{\log^2(n)}  \frac{\log^2(k)}{k}\right)
=
O\left((\log\log(n) )^3 \right)
\end{align*}
and thus can be neglected with respect to the scale of the problem. Now consider the sum 
\begin{align*}
 \sum_{k=1}^{\log^2(n)} \Lambda(k)( e^{-x_k} -1+x_k) 
 \quad \text{ with   } x_k\coloneqq \frac{\vartheta}{k}\log(b_n).
\end{align*}
Since $e^{-x} -1+ x = O(x^2) $ as $x\to 0$, \eqref{eq:intro_Mangoldt_sum3} yields
\begin{align*}
 \sum_{k=1}^{\log^2(n)} \Lambda(k)( e^{-x_k} -1+x_k) 
= 
 \sum_{k=1}^{\infty} \Lambda(k)( e^{-x_k} -1+x_k) +O(1).
\end{align*}

Recall that the Mellin transform of the function $e^{-x}$ is $\Gamma(s)$ for $\Re(s)>0$. 
Then the inverse Mellin transform gives
\begin{align}
\label{eq:Mellin_GAMMA}
e^{-x} -1+ x = \frac{1}{2\pi i}\int_{c-i\infty}^{c+i\infty} \Gamma(s) x^{-s} \, ds
\end{align}
for $-2< c <-1$. Details about the Mellin transform can be found for instance in \cite{DuFlGo95}, but here we will only need \eqref{eq:Mellin_GAMMA}.
Then
\begin{align*}
 \sum_{k=1}^{\infty} \Lambda(k)( e^{-x_k} -1+x_k)  
=
\frac{1}{2\pi i}\int_{c-i\infty}^{c+i\infty} \Gamma(s)\bigl(\vartheta\log(b_n)\bigr)^{-s} \sum_{k=1}^{\infty} \Lambda(k)k^s  \, ds.
\end{align*}
We need to justify the change of the order of summation and integration. Notice that on the line of integration 
\begin{align*}
\left| \bigl(\vartheta\log(b_n)\bigr)^{-s} \sum_{k=1}^{\infty} \Lambda(k)k^s \right|
\leq 
\bigl(\vartheta\log(b_n)\bigr)^{-c} \sum_{k=1}^{\infty} \Lambda(k)k^c < \infty
\end{align*}
holds and thus the change of order is valid by dominated convergence.
Denote by $\sum_p$ the sum over all prime numbers. It then follows by the definition of the von Mangoldt function $\Lambda$, see \eqref{eq:intro_Mangoldt}, that we have for $\Re(s)<-1$
\begin{align*}
\sum_{k=1}^{\infty} \Lambda(k)k^s
=
\sum_{p} \log(p) \sum_{j=1}^\infty p^{js}
=
\sum_{p} \log(p) \frac{p^s}{1-p^s}
=
-\frac{\zeta'(-s)}{\zeta(-s)},
\end{align*}
where $\zeta(s)$ denotes the Riemann zeta function. The last equality can easily be deduced form the Euler product formula of $\zeta(s)$. Therefore, 
\begin{align*}
 \sum_{k=1}^{\infty} \Lambda(k)( e^{-x_k} -1+x_k)  
=
-\frac{1}{2\pi i}\int_{c-i\infty}^{c+i\infty} \Gamma(s)\bigl(\vartheta \log(b_n)\bigr)^{-s}\frac{\zeta'(-s)}{\zeta(-s)} \, ds.
\end{align*}
Apply now the residue theorem to shift the line of integration to $1/2+ iy$ with $y\in\R$, which gives a double pole at $s=-1$ and simple pole at $s=0$ and at the zeros of the zeta function. This yields
\begin{align*}
\sum_{k=1}^{\infty} \Lambda(k)( e^{-x_k} -1+x_k)  
& =
\vartheta \log(b_n)\bigl( 1-\log(\vartheta \log(b_n)) \bigr)
-
\sum_\rho \Gamma(-\rho) (\vartheta\log(b_n)\bigr)^\rho \\
&
-\log(2\pi) + O\left((\log(b_n))^{-\frac{1}{2}}\right).
\end{align*}
This completes the proof.
\end{proof}

It remains to prove Theorem~\ref{thm:classF_E_logO_n_tilde}. Recall that
$\log \widetilde{O}_n = \log \widetilde{Y}_n - \widetilde{\Delta}_n $
and that $\mathbb{E}_{\Theta}[\log \widetilde{Y}_n]$ was computed in Lemma~\ref{lem:E_tilde_Yn}. Unfortunately, the estimate given in \eqref{eq:classF_Delta_tilde} is not strong enough to deduce  Theorem~\ref{thm:classF_E_logO_n_tilde}, so that we need to compute $\mathbb{E}_{\Theta}[\widetilde{\Delta}_n]$ more precisely. We need to study the behavior of $\widetilde{D}_{nk}$ and $\widetilde{D}_{nk}^*$, which are defined in \eqref{eq:log_tildeYn_von_mangold} and \eqref{eq:log_tildeOn_von_mangold}. 
\begin{lemma}
\label{lem:classF_generating_uDnk_tilde}
For $k\in\N$ and $u\in\C$ the following holds:
\begin{enumerate}
\item[\textup{(1)}]  $ \displaystyle \ET{u^{\widetilde{D}_{nk}}} 
= 
 \frac{1}{h_n} \nth{\exp\left(g_\Theta(t) +(u-1) \widetilde{g}_{\Theta,k}(t) \right)},$ \\
\item[\textup{(2)}]  $ \displaystyle \ET{\widetilde{D}_{nk}} 
= \frac{1}{h_n} \nth{ \widetilde{g}_{\Theta,k}(t)\exp\left(g_\Theta(t) \right)},$\\
\item[\textup{(3)}]  $ \displaystyle \mathbb{P}_{\Theta}\big[\widetilde{D}_{nk}^*=0 \big]
    =\frac{1}{h_n} \nth{\exp\left(g_\Theta(t)- \widetilde{g}_{\Theta,k}(t) \right)},$
\end{enumerate}
where
\begin{align}
\label{eq:def:g_T_k_tilde}
 \widetilde{g}_{\Theta,k}(t)
=
\sum_{m=1}^{b_n} \frac{\theta_m}{m}\one_{\{k|m\}} t^m.
\end{align}
%
\end{lemma}

\begin{proof}
Equation $(1)$ follows with a similar computation as in the proof of Lemma~\ref{lem:classF_generating_E_log_Y_tilde} and we thus omit it. Assertion
$(2)$ then follows from $(1)$ by differentiation with respect to $u$ and substituting $u=0$ and $(3)$ by substituting $u=0$ in $(1)$.
\end{proof}

The previous lemma implies
\begin{lemma}
\label{lem:classF_asymptotic_Dk_tilde}
Let $g_{\Theta}(t) \in \mathcal{F}(r,\vartheta,K)$. We then have for $2\leq k\leq n$ 
%
%

\begin{enumerate}
\item[\textup{(1)}]  $ \displaystyle \ET{\widetilde{D}_{nk}}
= \sum_{m=1}^{b_n}\frac{\theta_m}{m}r^m\one_{\{k|m\}} + O\Big(\frac{b_n}{nk}\Big),$ \\
\item[\textup{(2)}]  $\displaystyle \PT{\widetilde{D}_{nk}^*=0}
= \exp\left(-\sum_{m=1}^{b_n}\frac{\theta_m}{m}r^m\one_{\{k|m\}}\right) + O\Big(\frac{b_n}{nk}\Big).$
\end{enumerate}
\end{lemma}

\begin{proof}
For $b_n< k\leq n$ we have $\widetilde{D}_{nk}\equiv \widetilde{D}^*_{nk} \equiv 0$ and thus equation (1) and (2) are valid. We thus only have to consider $2\leq k\leq b_n$.
The proof is very similar to the proof of Lemma~\ref{lem:E_tilde_Yn}, including the contour of integration. 
One only has to replace $\widetilde{q}_1(t)$ by $\widetilde{g}_{\Theta,k}(t)$ and to use 
\begin{align*}
\widetilde{g}_{\Theta,k}\left(1+\frac{w}{n}\right)
=
\sum_{m=1}^{b_n}\frac{\theta_m}{m}r^m\one_{\{k|m\}} + O\Big(\frac{w b_n}{n k}\Big)
\end{align*}
for $w= O(\log^2(n))$. All other computations are identical and we thus omit them.
\end{proof}
%
%

\begin{proof}[Proof of Theorem~\ref{thm:classF_E_logO_n_tilde}]
Lemma~\ref{lem:E_tilde_Yn} gives us the behavior of $\mathbb{E}_{\Theta}[\log \widetilde{Y}_n]$.
It is thus enough to compute the expected value of $\widetilde{\Delta}_n = \log\widetilde{Y}_n-\log\widetilde{O}_n$. Equations  \eqref{eq:log_tildeYn_von_mangold} and \eqref{eq:log_tildeOn_von_mangold} yield
\begin{align}\label{eq:exp_logO_tilde_Delta_tilde}
\ET{\widetilde{\Delta}_n} 
=
\sum_{k=1}^{n} \Lambda(k)\ET{\widetilde{D}_{nk} - \widetilde{D}_{nk}^*}.
\end{align}
Denote $\alpha \coloneqq \lfloor \log^2 (n) \rfloor$ and consider the two sets 
$S_1\coloneqq\{1\leq k \leq \alpha\}$ and $S_2\coloneqq\{\alpha<k\leq n\}$. 
We split the sum according to the two sets and show first that the second sum is negligible.
Indeed, by Proposition~\ref{prop:bound_Dnk} and \eqref{eq:intro_Mangoldt_sum3},
\begin{align*}
 \sum_{k \in S_2 } \Lambda(k) \ET{\widetilde{D}_{nk} - \widetilde{D}_{nk}^*}
&=\,
O\left( \sum_{k\in S_2 } \Lambda(k)\ET{\widetilde{D}_{nk}(\widetilde{D}_{nk}-1)}  \right)\nonumber\\
&=\,
O\left( \sum_{k\in S_2 } \Lambda(k)\ET{D_{nk}(D_{nk}-1)}  \right)\nonumber\\
&=
O\bigg(\log^2(n) \sum_{k\in S_2} \frac{\Lambda(k)}{k^2}  \bigg)
=O(1).
\end{align*}
It is thus sufficient to consider the sum over the set $S_1$. 
Lemma~\ref{lem:classF_asymptotic_Dk_tilde} then yields for $k\leq \log^2 (n)$
\begin{align*}
&\ET{\widetilde{D}_{nk} - \widetilde{D}_{nk}^*} 
=
\ET{\widetilde{D}_{nk}} -1+ \PT{\widetilde{D}_{nk}^*=0} \\
=&
\exp\left(-\sum_{m=1}^{b_n}\frac{\theta_m}{m}r^m\one_{\{k|m\}}\right) -1+ \sum_{m=1}^{b_n}\frac{\theta_m}{m}r^m\one_{\{k|m\}}  + \,O\Big(\frac{b}{nk}\Big).
\end{align*}
Since $\Lambda(k) \leq \log(k)$, the sum over the error term is of order
\begin{align*}
&\frac{b_n}{n}\sum_{k\in S_1} O\left(\frac{\Lambda(k)}{k}\right)
=
O\left(\frac{(\log\log(n))^2}{\log^2(n)} \right)
=
O\left(\frac{1}{\log(n)} \right).
\end{align*}
Altogether, we proved that 
\begin{align*}
 \ET{\widetilde{\Delta}_n}
 =
 \sum_{k=1}^{\log^2(n)} \Lambda(k) \left(e^{-\sum_{m=1}^{b_n}\frac{\theta_m}{m}r^m\one_{\{k|m\}}} -1+ \sum_{m=1}^{b_n}\frac{\theta_m}{m}r^m\one_{\{k|m\}}\right)
 +
 O(1).
\end{align*}
Using the definition of $\Delta_n$ and Lemma~\ref{lem:E_tilde_Yn} completes the proof.
\end{proof}

\vskip 15pt 
\subsection{Expected value of $\log O_n$}
\label{subsection:classF_E_On}

We provide in this section a precise expansion of the expected value of $\log O_n$ which has in particular an interpretation in terms of the Riemann hypothesis. 
In this section we require additional assumptions on the function $g_{\Theta}$, namely that $g_{\Theta} \in \mathcal{LF}(r,\vartheta)$, which will be defined in Definition~\ref{def:Delta0K}. For this class of functions we will prove the following
\begin{theorem}\label{thm:classF_E_logO_n}
Suppose that $g_\Theta\in \mathcal{LF}(r,\vartheta)$. Then
\begin{align}
 \ET{\log O_n}
&= \ET{\log Y_n}
-
\vartheta \log(n)\bigl( 1-\log(\vartheta \log(n)) \bigr)\nonumber\\
&+
\sum_\rho \Gamma(-\rho) (\vartheta\log(n)\bigr)^\rho
+
O\left((\log\log(n))^3 \right).
\label{eq:E_On_exact}
\end{align}
\end{theorem}
This statement yields as an immediate consequence
\begin{corollary}
\label{cor:Equiv_RH}
Suppose that $g_\Theta\in \mathcal{LF}(r,\vartheta)$. Then following statements are equivalent 
\begin{enumerate}
 \item[\textup{(1)}]  The Riemann hypothesis is true.
 \item[\textup{(2)}]  We have for all $\epsilon>0$
\begin{align*}
  \ET{\log O_n}
=\ET{\log Y_n}
- \vartheta \log(n)\bigl( 1-\log(\vartheta \log(n)) \bigr)
+O\left(\log(n)\right)^{\frac{1}{2}+\epsilon}.
\end{align*}
\end{enumerate}
\end{corollary}

Equation \eqref{eq:E_On_exact} was proven by Zacharovas in \cite{Za04a} for the uniform measure on $\Sn$ and in \cite{Za04b} on the 
subgroup $\mathfrak{S}_n^{(k)} \coloneqq \{ \sigma = \tau^k|\tau\in\Sn\}$. Zacharovas also noted the implication $(1) \Rightarrow (2)$ of 
Corollary~\ref{cor:Equiv_RH}, but not the important opposite implication.

Recall that the crucial point in the proof of Theorem~\ref{thm:classF_E_logO_n_tilde} was the expansion of $\mathbb{E}_{\Theta}[\widetilde{\Delta}_n]$ as in \eqref{eq:exp_logO_tilde_Delta_tilde} and the expected values of $\widetilde{D}_{nk}$ and $\widetilde{D}^*_{nk}$ for $k\leq \log^2(n)$.
We thus start by studying  $\ET{D_{nk}}$ and $\ET{D^*_{nk}}$. 
\begin{lemma}\label{lem:classF_generating_uDnk}
For $k\in\N$ and $u\in\C$ the following holds:
\begin{enumerate}
\item[\textup{(1)}]  $ \displaystyle \ET{u^{D_{nk}}} 
=  \frac{1}{h_n}t^n [ \exp\left(g_\Theta(t) +(u-1) g_{\Theta,k}(t) \right)], $ \\
\item[\textup{(2)}]  $ \displaystyle \ET{D_{nk}} 
= \frac{1}{h_n}t^n [ g_{\Theta,k}(t)\exp\left(g_\Theta(t) \right)],$\\
  \item[\textup{(3)}]  $ \displaystyle \PT{D_{nk}^*=0} 
  =
  \frac{1}{h_n}t^n [\exp\left(g_\Theta(t)- g_{\Theta,k}(t) \right)],$
\end{enumerate}
where
\begin{align}
\label{eq:def:g_T_k}
 g_{\Theta,k}(t)
=
\sum_{m=1}^\infty \frac{\theta_m}{m}\one_{\{k|m\}} t^m.
\end{align}
%
%
\end{lemma}

\begin{proof}
 %
%
The proof is very similar to the proof of Lemma~\ref{lem:classF_generating_uDnk_tilde}. 
\end{proof}

Equation \eqref{eq:classF_theta_k_to_vartheta} implies that that there exists constants $c,C>0$ such that  
$c\,\theta_m r^m \leq \vartheta \leq C \,\theta_m r^m$ for $m$ large
if $g_\Theta\in\mathcal{F}(r,\vartheta,K)$.
Thus $g_{\Theta,k}$ has radius of convergence $r$ for all $k$. If we would like to use a similar argument as in 
Lemma~\ref{lem:classF_asymptotic_Dk_tilde}, we require further assumptions on the function $g_\Theta$. To get a vague intuition, let us have a look at the Ewens measure, meaning that $\theta_m=\vartheta$ for all $m\in\N$.
For this model, 
\begin{align*}
 g_\Theta(t) = \vartheta\log\left( \frac{1}{1-t/r} \right) 
 \quad \text{ and } \quad g_{\Theta,k}(t) = \frac{\vartheta}{k}\log\left( \frac{1}{1-(t/r)^k} \right).
\end{align*}
Clearly, each $g_{\Theta,k}(t) $ can be extended beyond its disk of convergence and its singularities are $k$-th roots of unity. 
These observations motivate the following definition.
\begin{definition}
\label{def:Delta0K}
Let $r,\vartheta>0$ be given. We write $\mathcal{LF}(r,\vartheta)$ for the set of all functions 
$g_\Theta(t) = \sum_{m=1}^\infty \frac{\theta_m}{m}t^m$ such that there exists $R>r$ and $0 < \phi <\frac{\pi}{2}$
so that the following conditions are satisfied for all $k\in\N$: 
\begin{enumerate}
 \item[\textup{(1)}]  $g_{\Theta,k}$ is holomorphic in $\Delta_{0,k}(r,R,\phi)\coloneqq\bigcap_{m=0}^{k-1} e^{\frac{2\pi m i }{k}}\Delta_0(r,R,\phi)$ (see Figure~\ref{fig:curve_flajolet_4_a}) with $g_{\Theta,k}$ as in \eqref{eq:def:g_T_k}.
 \item[\textup{(2)}]  We have 
   \begin{align}
   g_{\Theta,k}(t) = \frac{\vartheta}{k} \log\left( \frac{1}{1-(t/r)^k} \right) + K_k+ O \left( t-r \right)   \text{ as } t\to r
   \label{eq_class_F_r_alpha_near_r_k}
   \end{align}
   with $O(\cdot)$ uniform in $k$ and $K_k= O(1/k)$.
%
\end{enumerate}
\end{definition}
%

We require for the the proof of Theorem~\ref{thm:classF_E_logO_n} the asymptotic behavior of $\ET{D_{nk}}$ and $\ET{D^*_{nk}}$ for $g_\Theta\in \mathcal{LF}(r,\vartheta)$.
We have

\begin{lemma}
\label{lem:classF_asymptotic_Dk}
Suppose that $g_{\Theta} \in \mathcal{LF}(r,\vartheta)$, then the following holds uniformly in $k$ for
$2\leq k\leq n^{\frac{\vartheta}{1+\vartheta}}$ :
\begin{enumerate}
\item[\textup{(1)}]  $\displaystyle\ET{D_{nk}}
= \frac{\vartheta}{k}\log\left(\frac{n}{k}\right) + O\Big(\frac{1}{k} + \frac{k^{\vartheta+1}}{n^\vartheta}\Big),$ \\

\item[\textup{(2)}]  $\displaystyle \PT{D_{nk}^*=0}
= \left(\frac{n}{k}\right)^{-\frac{\vartheta}{k}} \frac{\Gamma(\vartheta)}{\Gamma\bigl(\vartheta(1-\frac{1}{k})\bigr)}
\Big(1 + O\Big(\frac{1}{n} +\frac{1}{k}+ \frac{k^{\vartheta+1}}{n^\vartheta}\Big)\Big).$
\end{enumerate}
\end{lemma}

\begin{proof}
The proof is very similar to the proof of Lemma~\ref{lem:E_tilde_Yn}. We combine
Theorem~\ref{lem:classF_generating_uDnk} and Cauchy's integral formula to obtain
\begin{align*}
h_n \ET{D_{nk}}
&=
\frac{1}{2\pi i} \int_{\gamma} g_{\Theta,k}(t) \exp\left( g_\Theta(t) \right) \, \frac{dt}{t^{n+1}},\\
h_n \ET{D^*_{nk}}
&=
\frac{1}{2\pi i} \int_{\gamma} \exp\left( g_\Theta(t) -  g_{\Theta,k}(t) \right) \, \frac{dt}{t^{n+1}}.
\end{align*}
By assumption, $g_{\Theta,k}$ is holomorphic in some domain $\Delta_{0,k}(r,R,\phi)$ (see Definition~\ref{def:Delta0K}).
Following the idea in \cite[Section~VI.3]{FlSe09},  we choose the curve $\gamma$ as in Figure~\ref{fig:curve_flajolet_4_a}, such that $\gamma$ is contained in $\Delta_{0,k}(r,R,\phi)$.

\begin{figure}[h]
\centering
\includegraphics[height=.22\textheight]{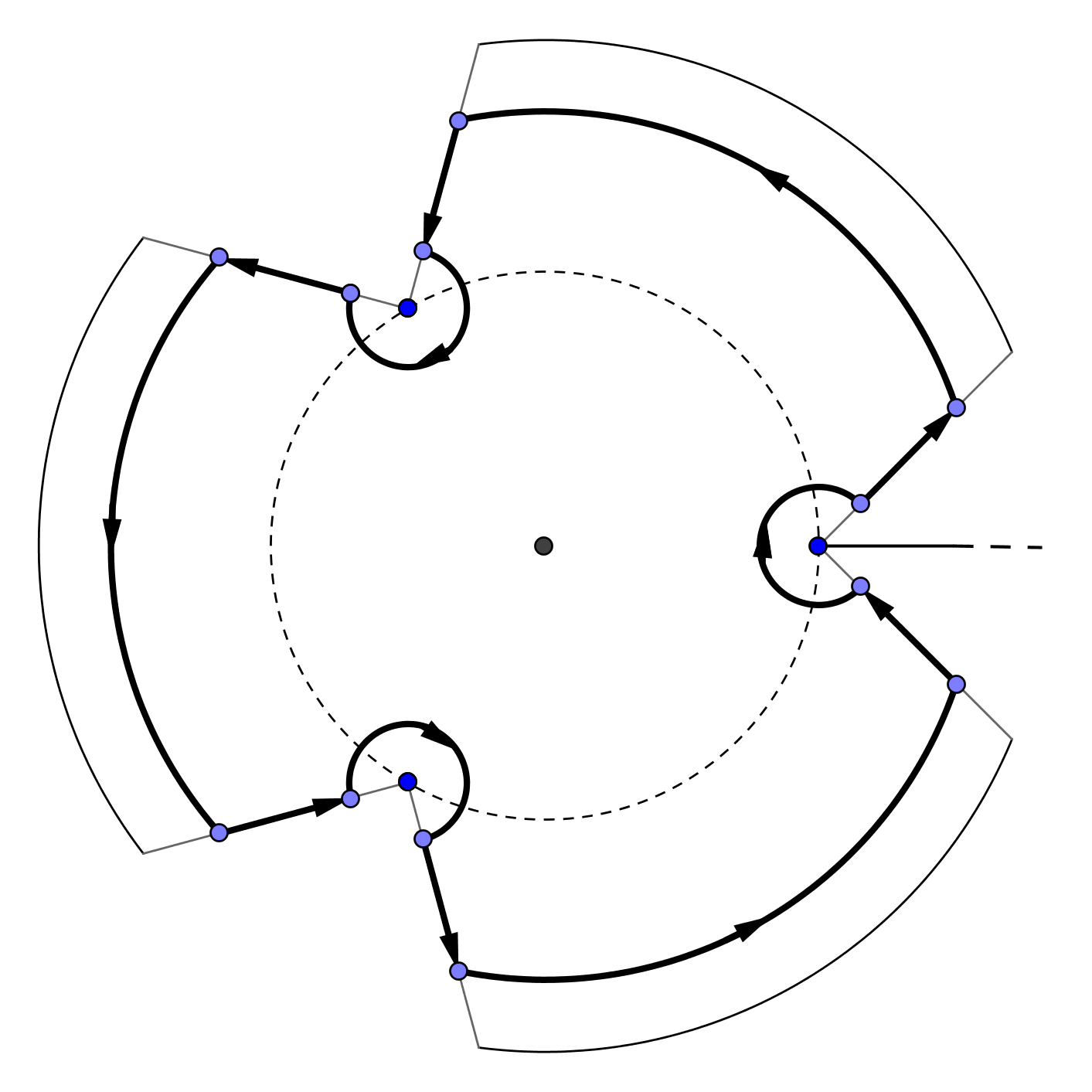}
 \put(-66,53,5){\mbox{\scriptsize$0$}}
  \put(-94,70){\mbox{\scriptsize$|z|= r$}}
 \put(-29,121){\mbox{\scriptsize$|z|= R$}}
 \put(-18,54){\mbox{\scriptsize$|z|= R'$}}
 \put(-115,42){\mbox{\scriptsize$\gamma$}}
\caption{{Illustration of the curve $\gamma$ in proof of Lemma~\ref{lem:classF_asymptotic_Dk}} with $k=3$.}
\label{fig:curve_flajolet_4_a}
\end{figure}
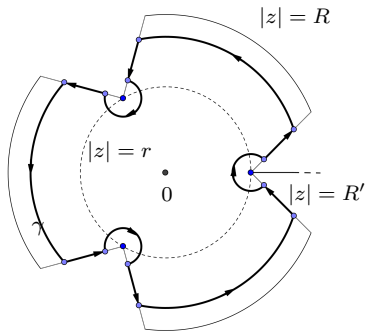
More precisely, we choose the radius of the big circle $\gamma_4\coloneqq\gamma_{4,0}\cup\cdots\cup\gamma_{4,k-1}$ as $R^\prime\coloneqq r(1+b_n^{-1})$ with $b_n$ as in \eqref{eq:def_On_tilde_Yn_tilde}, 
the radii of the small circles as $1/n$ and the angles of the lines segments all equal and independent of $n$.

Let us first show that the integral over the big circle $\gamma_4$ can be neglected. 
Since $kb_n=o(n)$, we get
%
%
\begin{align*}
 |g_{\Theta,k}(t)| \leq \vartheta \log \left|\frac{1}{1-(t/r)^k}\right| + O(1)
 \, \implies \,
|g_\Theta(R^\prime e^{i\varphi})| 
\leq
\vartheta \log(kb_n) +O(1).
\end{align*}
The estimates on $t^{-n-1}$ and $g_\Theta(t)$ are the same as in the proof of Lemma~\ref{lem:E_tilde_Yn}.
Combining all three, one immediately realizes that the integral over $\gamma_4$ is negligible.

It remains to compute the behavior along the curves around the points $r\cdot e^{j\frac{2\pi i}{k}}$ for $0\leq j < k$.
We have to distinguish the cases $1\leq j <k$ and $j=0$. For $j=0$ use the variable substitution $t = r(1+w/n)$ with $w=O(\log^2(n))$. 
This maps the curve around $r$ to the bounded curve $\gamma'$ in Figure~\ref{fig:gamma'}.
Furthermore, on this curve the following expansions hold:
\begin{align*}
 g_\Theta(t)     &= \vartheta \log(n)-\vartheta\log(-w) +K+ O(w/n),\\
 g_{\Theta,k} (t)&= \frac{\vartheta}{k}\bigl(\log(n/k) -\log(-w)\bigr)+ K_k + O(w/n),\\
 t^{-n-1} &= r^{-n-1} e^{-w}(1+O(w/n))  .
\end{align*}

 This implies
\begin{align*}
&\frac{1}{2\pi i} \int_{\gamma_{1,0}\cup\gamma_{2,0}\cup\gamma_{3,0}} \exp\left( g_\Theta(t) -g_{\Theta,k}(t) \right) \, \frac{dt}{t^{n+1}}\\
=&
\frac{n^{\vartheta(1-\frac{1}{k})-1} e^{K-K_k}}{r^n k^{\vartheta\frac{-1}{k}} 2\pi i} 
\int_{\gamma'} (-w)^{-\vartheta(1-\frac{1}{k})} e^{-w} (1+O(w/n)) \ dw.
\end{align*}

%
As in the proof of Lemma~\ref{lem:E_tilde_Yn}, one can replace the bounded curve $\gamma'$ by the Hankel contour $\gamma''$ in Figure~\ref{fig:gamma''}.
Using again \eqref{eq:classF_integral_Gamma} and Corollary~\ref{cor:haviour_of_hn} shows that the integral over this part gives the 
main term in Equation (2) of Lemma~\ref{lem:classF_asymptotic_Dk}. The argument for (1) is similar.

We now proceed to  $1\leq j \leq k-1$. We use here the variable substitution $t = r\cdot e^{j\frac{2\pi i}{k}}(1+w/n)$.
The curve  $\gamma_j\coloneqq\gamma_{1,j}\cup\gamma_{2,j}\cup\gamma_{3,j}$ is also mapped to $\gamma'$, but here the expansions along $\gamma'$ are given by
%
%
\begin{align*}
g_\Theta(t) 
&= 
\vartheta \log\left(  1 -e^{\frac{2\pi i j}{k}} \right)
+ O\Big(\frac{wk}{n}\Big) +O(1),\\
g_{\Theta,k}(t) 
&= 
\frac{\vartheta}{k}\bigl(\log(n/k) -\log(-w) \bigr)+K_k +O(w/n),\\
t^{-n-1} 
&= 
r^{-n-1}e^{\frac{2\pi i j}{k}}e^{-w}(1+O(w/n))  .
\end{align*}
Insert this into the Cauchy integral and summing over $j$ from $1$ to $k-1$ gives the error terms in (1) and (2).
\end{proof}

We are now prepared to prove the main result of this section.
\begin{proof}[Proof of Theorem~\ref{thm:classF_E_logO_n} ]
The argument is very similar to the one of proof of the Theorem~\ref{thm:classF_E_logO_n_tilde} and Corollary~\ref{cor:classF_E_logO_n_tilde_with_RH} . We thus give here only a short overview.
Recall that
\begin{align*}
\ET{\log Y_n}-\ET{\log O_n}
= 
\ET{\Delta_n} 
=
\sum_{k=1}^{n} \Lambda(k)\ET{D_{nk} - D_{nk}^*}.
\end{align*}
Denote $\alpha \coloneqq \lfloor \log^2 (n) \rfloor$ and consider the two sets 
$S_1\coloneqq\{1\leq k \leq \alpha\}$ and $S_2\coloneqq\{\alpha<k\leq n\}$. As in the proof of Theorem~\ref{thm:classF_E_logO_n_tilde}, we can show that the sum over the second set is negligible.
%
%
%
It is thus sufficient to consider only the sum over $S_1$. Lemma~\ref{lem:classF_asymptotic_Dk} yields for $k\leq \log^2 (n)$
\begin{align*}
&\ET{D_{nk} - D_{nk}^*} 
=
\ET{D_{nk}} -1+ \PT{D_{nk}^*=0} \\
&=
\frac{\vartheta}{k}\log (n) -1+ \left(\frac{n}{k}\right)^{-\frac{\vartheta}{k}}  \left(1+ O\left(\frac{1}{k} + \frac{k^{\vartheta+1}}{n^\vartheta}\right) \right) 
+ O\left(\frac{\log (k)}{k}\right)\nonumber\\
&=
\frac{\vartheta}{k}\log(n)-1  + e^{-\frac{\vartheta}{k}\log(n)} + O\left(\frac{\log(k)}{k}\right).
\end{align*}
This is now (almost) the same expression as in the proof of Corollary~\ref{cor:classF_E_logO_n_tilde_with_RH}. The remaining computations are the same and thus we omit them.
\end{proof}

\vskip 40pt 

\section{Parameters with polynomial growth: $\theta_m = m^{\gamma}, \gamma >0$ } \label{section:algrowth}

Now we turn our attention to a different class of parameters, namely polynomial parameters $\theta_m = m^{\gamma}$ with $\gamma >0$. Only few results are known for these parameters. Ercolani and Ueltschi \cite{ErUe11} show that for this model, a typical cycle has length of order $n^{\frac{1}{1+\gamma}}$ and that the total number of cycles has order $n^{\frac{\gamma}{1+\gamma}}$. Recently, we proved in \cite{StZe14a} that the cycle counts of the small cycles of length of order $o(n^{\frac{1}{1+\gamma}})$ can be approximated by independent Poisson random variables. Using this result, we proved the Erd\"os-Tur\'an law for this setting, see \cite[Theorem 4.3]{StZe14a}. 

In this section we will prove large deviations estimates for $\log O_n$. The method we are applying to get our results is the saddle-point method. We will not repeat all details about this method here and refer the reader to Section 2.3 and Section 4.1 in \cite{StZe14a}.
%

\vskip 15pt 
\subsection{Preliminaries}\label{subsection:algrowth_preliminaries}

As in Section~\ref{section:classF}, our basic strategy is to establish results for the approximating random variable $\log Y_n = \sum_{m=1}^n \log(m) C_m$ and then to show that $\Delta_n \coloneqq \log Y_n - \log O_n$ is small enough to transfer the result to $\log O_n$.
Recall \eqref{eq:generating_series_logYn_allgemein}, then for parameters $\theta_m = m^{\gamma}$ the generating series of $\log Y_n$ can the be written as
\begin{align}\label{eq:algrowth_g_Theta}
 \sum_{n=0}^{\infty} h_n \E_{\Theta}[\exp(s \log Y_n)] t^n  
= \exp\left( \sum_{m=1}^{\infty} \frac{1}{m^{1 - s - \gamma}} t^m  \right) =: \exp (\hat{g}_{\Theta}(t,s)).
\end{align}
As we consider $s$ fixed for the moment, we may write $\hat{g}_{\Theta}(t)$ instead of $\hat{g}_{\Theta}(t,s)$.
The function $\hat{g}_{\Theta}$ is known to be the polylogarithm $\Li_{\alpha}$ with parameter $\alpha = 1 - s -\gamma$. Its radius of convergence is $1$ and as $t \rightarrow 1$ it satisfies the following asymptotics for $\alpha \notin \{1, 2, ... \}$
\begin{align}\label{eq:algrowth_asypm_polylog}
 \Li_{\alpha}(t)
\sim \Gamma(1-\alpha)(-\log (t))^{\alpha -1} + \sum_{j \geq 0} \frac{(-1)^j}{j!}\xi(\alpha - j) (-\log (t))^j .
\end{align}
For $\gamma > 0$, that is $\alpha < 1$, this implies
\begin{align*}
 \hat{g}_{\Theta}(t) 
=  \Gamma(1-\alpha)(-\log (t))^{\alpha-1} +\zeta(\alpha) +O(t-1)
\end{align*} 
and an appropriate method to investigate the behavior of $\hat{g}_{\Theta}$ is the saddle-point method.
In \cite[Lemma 4.1]{StZe14a} we show that $\hat{g}_{\Theta}$ is log-admissible (see Definition 2.8 in \cite{StZe14a}). This gives us the asymptotic behavior of $h_n$:
\begin{align}
h_n 
&= \big(2\pi \Gamma(2+\gamma)\big)^{-\frac{1}{2}} \Big(\frac{\Gamma(1+\gamma)}{n} \Big)^{\frac{2+\gamma}{2(1+\gamma)}}  \times \nonumber\\
&\exp\bigg(n^{\frac{\gamma}{1+\gamma}}\bigg(\Gamma(1+\gamma)^{\frac{1}{1+\gamma}} + \frac{\Gamma(\gamma)}{\Gamma(1+\gamma)^{\frac{\gamma}{1+\gamma}}}\bigg) \bigg)\big(1+o(1)\big) 
\label{eq:algrowth_h_n}
\end{align}
and an expression for the generating function of $\log Y_n$:
\begin{theorem}[\cite{StZe14a}, Theorem 4.5]
\label{thm:algrowth_mom-gen_logYn}
Let $\hat{g}_{\Theta}$ be as in \eqref{eq:algrowth_g_Theta} with $\gamma > 0$. Then we have
\begin{align*}
&\E_{\Theta}[\exp(s \log Y_n)] \\
= & \Big(\sqrt{\tilde{\gamma}_{2,s}} \,  n^{\frac{1}{2}(\frac{1}{1+\gamma} - \frac{1}{1+\gamma+s})}\Big) 
\exp\bigg(\tilde{\gamma}_{1,s} \, n^{1 -\frac{1}{1+\gamma+s}} -\tilde{\gamma}_{1,0} \, n^{1 -\frac{1}{1+\gamma}} \bigg) \big(1+o(1)\big)
\end{align*}
with
\begin{align*}
\tilde{\gamma}_{1,s} = \frac{(1+\gamma +s)\Gamma(\gamma +s)}{\Gamma(1+\gamma+s)^{1-\frac{1}{1+\gamma+s}}} , 
\quad 
\tilde{\gamma}_{2,s} =  \frac{(1+\gamma) \Gamma(1+\gamma +s)^{\frac{1}{1+\gamma+s}}}{(1+\gamma +s) \Gamma(1+\gamma)^{\frac{1}{1+\gamma}}},
\end{align*}
where the error bounds are uniform in $s$ for bounded $s$.
\end{theorem}
Similarly to Lemma~\ref{lem:classF_closeness1 log Y_n and log O_n}, we need an estimate for the closeness of $\log O_n$ and $\log Y_n$. 
This is given by the following
\begin{lemma}[\cite{StZe14a}, Lemma 4.6]
\label{lem:algrowth_closeness}
For $\theta_m = m^{\gamma}$ with $0<\gamma<1$ the following holds as $n \rightarrow \infty$:
\begin{align*}
 \mathbb{P}_{\Theta}\big(\Delta_n \geq \log(n)\log\log(n) \big) \rightarrow 0. 
\end{align*}
\end{lemma}

 \begin{remark}
  Notice that that Lemma~\ref{lem:algrowth_closeness} is wrong for $\gamma>1$, see \cite[Remark~4.5]{StZe14a}.
 \end{remark}

\vskip 15pt 
\subsection{Large deviations estimates for $\log O_n$}\label{subsection:algrowth_large_deviation}

From the moment generating function of $\log Y_n$ given in Theorem~\ref{thm:algrowth_mom-gen_logYn} 
we can deduce a classical large deviations result for $\log O_n$. We will show that for any Borel set $B$ 
\begin{align*}
\lim_{n \rightarrow \infty} n^{-\frac{\gamma}{1+\gamma}} \log\mathbb{P}_{\Theta}\bigg(\frac{\log O_n}{n^{\frac{\gamma}{1+\gamma}} \log(n)} \in B\bigg) 
= -\inf_{x\in B} \chi^*(x)
\end{align*}
holds, where
\begin{align*}
\chi^*(x) = \sup_{t \in \R}[t x - \chi(t)]
\end{align*}
is the so-called Fenchel-Legendre transform of $\chi(t)$ given by
\begin{align}\label{eq:algrowth_convex_dual}
\chi(t) \coloneqq \frac{(1+\gamma)\Gamma(\gamma)}{\Gamma(1+\gamma)^{\frac{\gamma}{1+\gamma}}}  \big(e^{\frac{t}{(1+\gamma)^2}}-1\big).
\end{align}
In other words, we will show the following
\begin{theorem}\label{thm:algrowth_large_dev_O_n}
Let $\hat{g}_{\Theta}$ be defined as in (\ref{eq:algrowth_g_Theta}) with $0<\gamma < 1$. The sequence $\log O_n / n^{\frac{\gamma}{1+\gamma}} \log(n)$ satisfies a large deviations principle with rate $n^{\frac{\gamma}{1+\gamma}} $ and rate function given by the convex dual of $\chi(t)$ defined in (\ref{eq:algrowth_convex_dual}).
\end{theorem}
\begin{proof}
 Let us first check that $\log Y_n / n^{\frac{\gamma}{1+\gamma}} \log(n)$ satisfies this large deviations estimate. By the G\"artner-Ellis theorem it suffices to prove
\begin{align}\label{eq:algrowth_Gaertner-Ellis}
\lim_{n \rightarrow \infty} n^{-\frac{\gamma}{1+\gamma}} \log \ET{\exp\Big(t \, \frac{\log Y_n}{\log(n)}\Big)}= \chi(t).
 \end{align}
In view of Theorem~\ref{thm:algrowth_mom-gen_logYn} we have to show that for $t^* = t / \log(n)$
\begin{align*}
\lim_{n \rightarrow \infty} n^{-\frac{\gamma}{1+\gamma}} \bigg(\tilde{\gamma}_{1,t^*} \, n^{1 -\frac{1}{1+\gamma+t^*}} -\tilde{\gamma}_{1,0} \, n^{1 -\frac{1}{1+\gamma}} \bigg) 
= \chi(t)
\end{align*}
holds with
\begin{align*}
\tilde{\gamma}_{1,t} =  \frac{(1+\gamma +t)\Gamma(\gamma +t)}{\Gamma(1+\gamma+t)^{1-\frac{1}{1+\gamma+t}}}  .
\end{align*}
This is true since $\Gamma(\gamma+x) = \Gamma(\gamma) + O(x)$ as $x \rightarrow 0$ and therefore
\begin{align}
\tilde{\gamma}_{1,t^*}
& = \tilde{\gamma}_{1,0} + O\big(\log^{-1}(n)\big), \label{eq:algrowth_classical_LDP_approx1}\\
n^{1 -\frac{1}{1+\gamma+t^*}} 
&= n^{1 -\frac{1}{1+\gamma}} \Big(1+ \sum_{k=1}^{\infty} \frac{t^k}{(1+\gamma)^{2k} k!} + O\big(\log^{-1}(n)\big)\Big) \nonumber\\
&= n^{\frac{\gamma}{1+\gamma}} \Big(e^{\frac{t}{(1+\gamma)^2}} + O\big(\log^{-1}(n)\big) \Big). \label{eq:algrowth_classical_LDP_approx2}
\end{align}
Similar the proof of Theorem~\ref{thm:classF_large_dev_logO_n}, it remains to show that $\log Y_n / n^{\frac{\gamma}{1+\gamma}} \log(n)$ and $\log O_n / n^{\frac{\gamma}{1+\gamma}} \log(n)$ are exponentially equivalent with rate $n^{\frac{\gamma}{1+\gamma}}$. This is subject of the following lemma.
\end{proof}
\begin{lemma}\label{lem:algrowth_exp_equiv_logYn_logOn}
Let $\hat{g}_{\Theta}$ be as in (\ref{eq:algrowth_g_Theta}) with $0<\gamma< 1$, then for any $c>0$ the following holds:
\begin{align*}
\limsup_{n \rightarrow \infty} n^{-\frac{\gamma}{1+\gamma}}
\log \PT{\log Y_n - \log O_n >c \, n^{\frac{\gamma}{1+\gamma}} \log(n)} = -\infty.
\end{align*}
\end{lemma}
\begin{proof}
We will prove a stronger version of this asymptotic in Lemma~\ref{lem:algrowth_exp_equiv_mod_pois_logYn_logOn}.
\end{proof}
The statement of Theorem~\ref{thm:algrowth_large_dev_O_n} can be refined. Recall the notion of mod-$\phi$ convergence which was briefly explained in Section~\ref{subsection:classF_mod_conv_local_limit}. Here, we prove mod-Poisson convergence for $\log Y_n$, appropriately rescaled, in terms of moment generating functions. We deduce the following precise deviations estimate for $\log O_n$:
\begin{theorem}\label{thm:algrowth_precise_dev}
Let $\hat{g}_{\Theta}$ be as in (\ref{eq:algrowth_g_Theta}) with $0<\gamma < 1$. Define 
\begin{align*}
\mathcal{O}_n \coloneqq \frac{(1+\gamma)^2\log O_n - \lambda_n \log(n)}{\lambda_n^{1/3}\log(n)}
\end{align*}
and
\begin{align*}
\lambda_n \coloneqq \tilde{\gamma}_{1,0} \, n^{\frac{\gamma}{1+\gamma}} (1+O(\log^{-1}(n))),
\end{align*}
where $\tilde{\gamma}_{1,0}$ is  as in Theorem~\ref{thm:algrowth_mom-gen_logYn}. Then for any $x>0$ the following asymptotic holds:
\begin{align*}
 \PT{\mathcal{O}_n \geq x\lambda_n^{1/3}} = \frac{\exp(-\lambda_n^{1/3} \frac{x^2}{2} + \frac{x^3}{6})}{x \lambda_n^{1/6} \sqrt{2\pi}}  \,  (1+o(1)).
\end{align*}
\end{theorem}

\begin{proof}
Let us first check that
\begin{align*}
\mathcal{Y}_n \coloneqq \frac{(1+\gamma)^{2}\log Y_n}{\log(n)}
\end{align*}
satisfies the required precise deviations estimate. Indeed, $\mathcal{Y}_n$ is mod-Poisson convergent with parameter $\lambda_n$  and limiting function $\psi(t) = e^{t/2}$, that is
$$ \lim_{n \rightarrow \infty} e^{-\lambda_n (e^t-1)} \ET{e^{t \mathcal{Y}_n}} = e^{t/2}.$$

This follows directly from the moment generating function of $\log Y_n$ together with \eqref{eq:algrowth_classical_LDP_approx1} and \eqref{eq:algrowth_classical_LDP_approx2}. 
Notice that this convergence is surprising since the rescaling by $\log(n)$ in $\mathcal{Y}_n$ is relatively insignificant compared to the order of $\log Y_n$ which 
is $n^{\frac{\gamma}{1+\gamma}} \log(n)$. This statement suggests that $\log Y_n$ is indeed close to a Poisson random variable. 
However, the rescaling is too small to deduce a Poisson behavior of $\log O_n$.
 
\begin{remark}
  We have computed the moment generating function of $\log Y_n$ in Theorem~\ref{thm:algrowth_mom-gen_logYn} only for $s$ reel. 
  However, we require for the mod-Poisson convergence of $\mathcal{Y}_n$ above and the mod-Gaussian convergence below that Theorem~\ref{thm:algrowth_mom-gen_logYn} is also 
  valid for complex values of $s$ for $s$ in a small neighbourhood of $0$. This is indeed true and can be proven complete similarly to Theorem~\ref{thm:algrowth_mom-gen_logYn}. 
  One only has to verify that the asymptotic behaviour of $\Li_{\alpha}$ in \eqref{eq:algrowth_asypm_polylog} is also valid for complex $\alpha$. This can be proven with precisely the same
  argumentation as for reel $\alpha$, see for instance \cite[Section~VI.8.]{FlSe09}. 
 \end{remark}

Now, similarly to the proof of Theorem~\ref{thm:classF_mod_large-dev1}, 
we want to apply Theorem 3.2 in \cite{FeMeNi13} in order to deduce the large deviations result. 
This theorem requires mod-$\phi$ convergence where the reference law is lattice distributed. Hence, we cannot work directly with the mod-Poisson convergence. 
However, notice that mod-Poisson convergence with growing parameters implies mod-Gaussian convergence:
\begin{align*}
\mathcal{\widetilde{Y}}_n \coloneqq \frac{\mathcal{Y}_n - \lambda_n}{\lambda_n^{1/3}}
\end{align*}
is mod-$\mathcal{N}(0,\lambda_n^{1/3})$ convergent with limiting function $\Phi(t) = e^{t^3/6}$. 
Now apply Theorem 3.2 in \cite{FeMeNi13} with $\beta_n = \lambda_n^{1/3}$, $F(x) = x^2/2 = \eta(x)$ and $h(x)=x$ to obtain that $\mathcal{\widetilde{Y}}_n$ satisfies the required estimate.

It remains to transfer the estimate to $\mathcal{O}_n$ as defined in Theorem~\ref{thm:algrowth_precise_dev}. Clearly, 
\begin{align*}
 \PT{\mathcal{O}_n \geq x\lambda_n^{1/3}}
 \leq  \PT{\mathcal{\widetilde{Y}}_n \geq x\lambda_n^{1/3}}.
\end{align*}
For the reverse direction, let $g$ be a positive function such that $g(n)=o(\lambda_n^{1/3})$. Then
\begin{align*}
 \PT{\mathcal{\widetilde{Y}}_n \geq x\lambda_n^{1/3} + g(n)}
 \leq  \PT{\mathcal{O}_n \geq x\lambda_n^{1/3}} + \PT{\Delta_n \geq g(n)\lambda_n^{1/3}\log(n)}
\end{align*}
holds and we also have
\begin{align*}
 \PT{\mathcal{\widetilde{Y}}_n \geq x\lambda_n^{1/3} + g(n)} =  \PT{\mathcal{\widetilde{Y}}_n \geq x\lambda_n^{1/3} } \big(1+o(1) \big). 
\end{align*}
Finally, to complete the proof we need to find an appropriate $g(n)=o(\lambda_n^{1/3})$ such that
\begin{align*}
\lim_{n \rightarrow \infty}\lambda_n^{-1/3} \log \PT{\Delta_n \geq g(n)\lambda_n^{1/3}\log(n)} = - \infty.
\end{align*}
The following lemma proves that this holds for $g(n)=n^{\frac{\gamma}{3(1+\gamma)}} / \sqrt{\log(n)}$.
\end{proof}

\begin{lemma}\label{lem:algrowth_exp_equiv_mod_pois_logYn_logOn}
Let $\hat{g}_{\Theta}$ be as in (\ref{eq:algrowth_g_Theta}) with $0<\gamma < 1$, then for any $c>0$ the following holds:
\begin{align*}
\lim_{n \rightarrow \infty} n^{-\frac{\gamma}{3(1+\gamma)}}
\log \PT{\log Y_n - \log O_n >c \, n^{\frac{2\gamma}{3(1+\gamma)}}  \sqrt{\log(n)}} = -\infty.
\end{align*}
\end{lemma}
\begin{proof}
The proof is very similar to the proof of Lemma~\ref{lem:algrowth_closeness}. Recall \eqref{eq:log_On_von_mangold} and notice that 
\begin{align*}
 \log O_n = \psi(n) - R(n)
\end{align*}
where
\begin{align*}
 \psi(n) = \sum_{k=1}^n \Lambda(k) \quad \text{and} \quad R(n) = \sum_{k=1}^n \Lambda(k)\one_{\{D_{nk}=0 \}}.
\end{align*}
Recall that $\psi$ is the so-called Chebyshev function as defined in \ref{eq:intro_Chebychev} which satisfies the asymptotic \ref{eq:Chebychev_asymp}.
First, we want to find the smallest $b$ such that 
\begin{align}\label{eq:algrowth_LDP_transfer_1}
n^{-\frac{\gamma}{3(1+\gamma)}} \log\PT{\log Y_n - \psi(b) > \frac{c}{2}n^{\frac{2\gamma}{3(1+\gamma)}} \sqrt{\log(n)}}
\rightarrow - \infty
\end{align}
and afterwards we show 
\begin{align}\label{eq:algrowth_LDP_transfer_2}
n^{-\frac{\gamma}{3(1+\gamma)}} \log\PT{R(n) - \sum_{k=b+1}^n \Lambda(k) > \frac{c}{2}n^{\frac{2\gamma}{3(1+\gamma)}} \sqrt{\log(n)}}
\rightarrow - \infty.
\end{align}
Theorem~\ref{thm:algrowth_mom-gen_logYn} implies a central limit theorem for $\log Y_n$ with mean $G(n) = O(n^\frac{\gamma}{1+\gamma}\log(n))$ and variance $F(n) = O(n^\frac{\gamma}{1+\gamma}\log^2(n))$, see Lemma 4.4 in \cite{StZe14a}. This tells us that that for 
$$x = \frac{\frac{c}{2} n^{\frac{2\gamma}{3(1+\gamma)}} \sqrt{\log(n)} + \psi(b) - G(n)}{\sqrt{ F(n)}}$$
 we get as $n \rightarrow \infty$
 \begin{align*}
 \PT{\log Y_n - \psi(b) \geq \frac{c}{2}n^{\frac{2\gamma}{3(1+\gamma)}} \sqrt{\log(n)} } 
 %
&=  \Big(1 - \frac{1}{2} \Big(1+\erf\Big(\frac{x}{\sqrt{2}} \Big) \Big)\Big)(1+o(1)).
\end{align*} 
Here, $\erf$ denotes the error function which satisfies the asymptotic
\begin{align*}
\erf(x) = 1 + O(x^{-1} e^{-x^2}) \quad \text { as } x \rightarrow \infty.
\end{align*}
Thus set $b = n^{\frac{\gamma}{1+\gamma}} \log(n)\alpha(n)$ for some function $\alpha \rightarrow \infty$ so that 
$$x = O\big(n^{\frac{\gamma}{2(1+\gamma)}}\alpha(n)\big)$$
where the error term has a positive sign. This implies   
 \begin{align*}
 n^{-\frac{\gamma}{3(1+\gamma)}} 
 \log\PT{\log Y_n - \psi(b)  > \frac{c}{2} n^{\frac{2\gamma}{3(1+\gamma)}} \sqrt{\log(n)}}
= O\Big(n^{-\frac{\gamma}{1+\gamma}} \log \big(x^{-1} e^{-x^2}\big)  \Big).
\end{align*} 
which converges indeed to $-\infty$ and hence \eqref{eq:algrowth_LDP_transfer_1} holds. So let us now prove \eqref{eq:algrowth_LDP_transfer_2}. Notice that
$$R(n) - \sum_{k=b+1}^n \Lambda(k) \leq R(b) \leq \sum_{k=1}^b \Lambda(k)\one_{\{C_k=0\}} =: S(b)$$
and therefore 
\begin{align*}
n^{-\frac{\gamma}{3(1+\gamma)}} \log \PT{S(b) > \frac{c}{2}n^{\frac{2\gamma}{3(1+\gamma)}}  \sqrt{\log(n)}}
\rightarrow - \infty
\end{align*}
implies \eqref{eq:algrowth_LDP_transfer_2}.
Via saddle point analysis we get 
\begin{align*}
\ET{e^{sS(b)}} = \exp\Big(\sum_{k=1}^b \log \big(1+(e^{s\Lambda(k)}-1)e^{-k^{\gamma-1}r^k} \big)\Big) \big(1+o(1) \big).
\end{align*}
We proceed as in the proof of Lemma 4.6 in \cite{StZe14a}. For any $s \geq 0$ Markov's inequality yields
\begin{align*}
&n^{-\frac{\gamma}{3(1+\gamma)}} \log \PT{S(b) > \frac{c}{2}n^{\frac{2\gamma}{3(1+\gamma)}} \sqrt{\log(n)}}\\
%
%
\leq & \, - \frac{sc}{2}n^{\frac{\gamma}{3(1+\gamma)}}\sqrt{\log(n)} + n^{-\frac{\gamma}{3(1+\gamma)}} \sum_{k=1}^b (e^{s\Lambda(k)}-1)e^{-k^{\gamma-1}r^k} \\
= & \, - \frac{sc}{2}n^{\frac{\gamma}{3(1+\gamma)}}\sqrt{\log(n)} +O\big(  n^{-\frac{\gamma}{3(1+\gamma)}}  (e^{s\log(n)}-1)\big).
\end{align*}
For the last equality notice that for $b = o\big(n^{\frac{1}{1+\gamma}}\big)$ (here we need the assumption $\gamma <1$), there is a constant $c>0$ such that
\begin{align*}
 \sum_{k=1}^b  \exp\big(-  t^k k^{\gamma-1}\big) 
&\leq \sum_{k=1}^b  \exp\big(-  k^{\gamma-1}\exp(- b n^{-\frac{1}{1+\gamma}})\big)\\
%
%
&\leq \int_1^b \exp\big(- c  x^{\gamma-1} ) dx\\
&= O\Big(\Gamma\Big(\frac{1}{1-\gamma},b^{\gamma-1}\Big) - \Gamma\Big(\frac{1}{1-\gamma},1\Big)\Big)\\
&= O(1).
\end{align*}
Now set $s = \log^{-1/2}(n)$ to get 
\begin{align*}
n^{-\frac{\gamma}{3(1+\gamma)}} \log \PT{S(b) > \frac{c}{2}n^{\frac{2\gamma}{3(1+\gamma)}}  \log(n)}
=  - \frac{c}{2} n^{\frac{2\gamma}{3(1+\gamma)}} + O\Big(  n^{-\frac{\gamma}{1+\gamma}} e^{\log^{1/2}(n)} \Big).
\end{align*}
The proof is complete.
\end{proof}

\subsection*{Acknowledgments}
The research leading to these results has been supported by the SFB $701$
(Bielefeld) and has received funding from the People Programme (Marie
Curie Actions) of the European Union's Seventh Framework Programme
(FP7/$2007-2013$) under REA grant agreement nr.$291734$.

\bibliographystyle{acm}
\bibliography{literatur}

\end{document}